\documentclass[12pt, reqno]{amsart}

\usepackage{amsmath}
\usepackage{amssymb,amsfonts,amscd}
\usepackage{latexsym}
\usepackage{mathrsfs}
\usepackage{color}

\usepackage{hyperref}

\newcommand\al{\alpha}
\newcommand\be{\beta}
\newcommand\ga{\gamma}

\newcommand\de{\delta}

\newcommand\om{\omega}
\newcommand\Om{\Omega}
\newcommand\la{\lambda}
\newcommand\La{{\Lambda}}
\newcommand{\eps}{\varepsilon}
\newcommand\si{\sigma}

\newcommand\R{\mathbb R}
\newcommand\C{\mathbb C}
\newcommand\Z{\mathbb Z}
\newcommand\Y{\mathbb Y}
\newcommand\E{\mathbb E}
\newcommand\D{\mathbb D}

\newcommand\MM{\mathbb M}
\newcommand\N{\mathbb N}
\newcommand\NN{\overline{\N}}
\newcommand\Ninfty{\NN{}^{\,\infty}}

\newcommand\sgn{\operatorname{sgn}}

\newcommand\Prob{\operatorname{Prob}}

\newcommand\Dom{\operatorname{Dom}}
\newcommand\Ran{\operatorname{Ran}}
\newcommand\Sym{\operatorname{Sym}}

\newcommand\Conf{\operatorname{Conf}}

\newcommand\B{\mathscr B}

\newcommand\abcd{{a,b; c,d}}
\newcommand\alde{{\al,\be;\ga,\de}}

\newcommand\s{\mathfrak s}

\newcommand\sms{\smallskip}

\newcommand\wt{\widetilde}
\newcommand\ov{\overline}

\newcommand\CT{\operatorname{CT}}
\newcommand\Anti{\mathfrak A}
\newcommand\x{\mathbf x}
\renewcommand\D{\mathcal D}
\newcommand\F{\mathcal F}

\newcommand\Iqab{{I_{q;a,b}}}

\newcommand\ccdot{\,\cdot\,}

\newcommand\m{v}

\newtheorem{theorem}{Theorem}[section]
\newtheorem{proposition}[theorem]{Proposition}
\newtheorem{lemma}[theorem]{Lemma}

\newtheorem{corollary}[theorem]{Corollary}

\theoremstyle{definition}
\newtheorem{definition}[theorem]{Definition}
\newtheorem{remark}[theorem]{Remark}

\numberwithin{equation}{section}

\begin{document}

\title[]{Infinite-dimensional $q$-Jacobi Markov processes}

\author{Grigori Olshanski}

\dedicatory{To the memory of Anatoly Moiseevich Vershik}

\date{}

\thanks{Supported by the Russian Science Foundation under project 23-11-00150}

\begin{abstract}
The classical Jacobi polynomials on the interval $[-1,1]$ are eigenfunctions of a second order differential operator. It is well known that this operator generates a diffusion process on $[-1,1]$. Further, this fact admits an extension to $N$ dimensions (Demni (2010), Remling-R\"osler (2011)) leading to a $3$-parameter family of diffusion processes $X_N$ on the space of $N$-particle configurations in $[-1,1]$. The generators of the processes $X_N$ are related to Heckman-Opdam's  Jacobi polynomials attached to the root system $BC_N$.  

The first result of the paper shows that the processes $X_N$ have a $q$-analog, the $N$-dimensional $q$-Jacobi processes. These are Feller Markov processes related to the $N$-variate symmetric big $q$-Jacobi polynomials. The later polynomials were introduced and studied by Stokman (1997) and Stokman-Koornwinder (1997); they depend on two Macdonald parameters $(q,t)$ and $4$ extra continuous parameters. 

The $N$-dimensional $q$-Jacobi processes are still defined on a space of $N$-particle configurations, only now the particles live not on $[-1,1]$ but on certain one-dimensional $q$-grids. 

The second result (the main one) asserts that the $N$-dimensional $q$-Jacobi processes survive a limit transition as $N$ goes to infinity and two of the extra parameters vary together with $N$ in a certain way. In the limit, one obtains a family of Feller Markov processes which are infinite-dimensional in the sense that they live on configurations with infinitely many particles.  The proof uses a lifting of the multivariate big $q$-Jacobi polynomials to the algebra of symmetric functions --- a construction that does not hold for the Heckman-Opdam's Jacobi polynomials.  

Note also that the large-$N$ limit transition is carried out without any space scaling, which would be impossible in the continuous case.  
\end{abstract}

\maketitle


\section{Introduction}

\subsection{Preliminaries}\label{sect1.1}

Introduce some necessary notation. Let $\Sym$ denote the algebra of symmetric functions over the base field $\R$. By a \emph{configuration} on $\R^*:=\R\setminus\{0\}$ we mean a finite or countable collection of points $X\subset \R^*$ such that the sum $\sum_{x\in X}|x|$ is finite. Let $\Conf(\R^*)$ denote the set of all such $X$'s. The elements $f\in\Sym$ can be evaluated at each  $X\in\Conf(\R^*)$, so that $\Sym$ can be regarded as an algebra of functions $f(X)$ on $\Conf(\R^*)$. 

The algebra $\Sym$ is graded. We fix two numbers $q$ and $t$ contained in the open interval $(0,1)$ and consider the homogeneous basis  in $\Sym$ formed by the Macdonald symmetric functions with the parameters $(q,t)$ (see \cite[ch. VI]{Mac-1995}). We denote them by $P_\la(-;q,t)$; here the index $\la$ ranges over the set of partitions (= Young diagrams), denoted by $\Y$. 

Next, we consider another basis in $\Sym$, which is inhomogeneous. It is formed by the so-called \emph{big $q$-Jacobi symmetric functions}, denoted by $\Phi_\la(-;q,t;\al,\be;\ga,\de)$. Here $\la$ ranges over $\Y$ as before, and $(\al,\be;\ga,\de)$ is a quadruple of extra parameters such that $\be<0<\al$ and $\ga=\bar\de\in\C\setminus\R$. The top degree homogeneous component of $\Phi_\la(-;q,t;\al,\be;\ga,\de)$ is $P_\la(-;q,t)$.

The basis $\{\Phi_\la(-;q,t;\al,\be;\ga,\de):\la\in\Y\}$ was defined in \cite[sect. 7]{Ols-2021b}. In the special case $q=t$ (when the Macdonald symmetric functions turn into the Schur symmetric functions), the definition simplifies; it was given earlier in \cite{Ols-2017}. 

We list a few key properties of the big $q$-Jacobi symmetric functions established in \cite{Ols-2017} and \cite{Ols-2021b}.

(i) There exists a linear functional
$$
\E^{q,t;\alde}: \Sym\to\R
$$
such that the symmetric bilinear form on $\Sym$ given by 
$$
(f,g):=\E^{q,t;\alde}(fg), \quad f,g\in\Sym,
$$
is a positive definite scalar product and the elements $\Phi_\la(-;q,t;\al,\be;\ga,\de)$ are orthogonal with respect to it.

(ii) Further, one can exhibit a subset $\wt\Om(q,t;\al,\be)\subset\Conf(\R^*)$ with the following properties:

\begin{itemize}

\item[(ii$_1$)] $\wt\Om(q,t;\al,\be)$ is a compact, totally disconnected space with respect to a natural topology;

\item[(ii$_2$)]  for any $f\in\Sym$, the corresponding function $f(X)$ is continuous on $\wt\Om(q,t;\al,\be)$, so that we may treat $\Sym$ as a subalgebra of $C(\wt\Om(q,t;\al,\be))$, the real Banach algebra of continuous functions on $\wt\Om(q,t;\al,\be)$; note also that $\Sym$ is dense in  $C(\wt\Om(q,t;\al,\be))$;

\item[(ii$_3$)]  $\wt\Om(q,t;\al,\be)$ is endowed with a probability measure $\MM^{q,t;\alde}$ such that the functional $\E^{q,t;\alde}$ is given by integration against $\MM^{q,t;\alde}$.

\end{itemize}

Note that (i) and (ii$_3$) are similar to fundamental properties of orthogonal polynomials, with $\MM^{q,t;\alde}$ playing the role of the orthogonality measure. 

The symmetric functions $\Phi_\la(-;q,t;\al,\be;\ga,\de)$ are an infinite-variate analog of the $N$-variate symmetric big $q$-Jacobi polynomials studied by Stokman and Koornwinder \cite{S}, \cite{SK}. Those polynomials are in turn a generalization of the classical univariate  big $q$-Jacobi polynomials of Andrews-Askey \cite{AA}. Infinite-variate analogs also exist for some other systems of $q$-hypergeometric orthogonal polynomials, see \cite{CO}. 

\subsection{The main result}\label{sect1.2}
As above, we fix a $6$-tuple of $(q,t;\alde)$ of parameters such that $q,t\in(0,1)$, $\be<0<\al$, and $\ga=\bar\de\in\C\setminus\R$. We introduce a linear operator $\D_\infty$ on the space $\Sym$, which has a diagonal form in the basis $\{\Phi_\la(-;q,t;\al,\be;\ga,\de):\la\in\Y\}$:
\begin{equation}
\D_\infty \Phi_\la(-;q,t;\al,\be;\ga,\de) = \m_\la(q,t;\alde) \Phi_\la(-;q,t;\alde),
\end{equation}\label{eq1.A}
where the eigenvalues $\m_\la(q,t;\alde)$ are given by the following formula (the sum below is actually finite):
\begin{equation}
\m_\la(q,t;\alde):=-\, \sum_{i\ge1}\left(\frac{\ga\de}{\al\be} t^{1-i}(q^{\la_i}-1)+t^{\la_i-1}(q^{-\la_i}-1)\right).
\end{equation}

The symmetric functions $\Phi_\la(-;q,t;\al,\be;\ga,\de)$ and the operator $\D_\infty$ are obtained, respectively, from the $N$-variate big $q$-Jacobi polynomials and the related second order partial $q$-difference operators $D_N$, as the result of a large-$N$ limit transition in which two of the parameters varies as $N\to\infty$.

Note that 
$$
\Phi_\varnothing(-;q,t;\al,\be;\ga,\de)=1, \qquad \m_\varnothing(q,t;\alde)=0,
$$
where $\varnothing$ denotes the null partition (= empty Young diagram). Therefore, $\D_\infty$ annihilates the constant functions. 
 
Since we can regard $\Sym$ as a dense subspace of the real Banach space $C(\wt\Om(q,t;\al,\be))$, we may treat $\D_\infty$ as a densely defined operator on that space. 

\begin{theorem}[see Theorem \ref{thm7.A}]\label{thm1.A}
The closure of\/ $\D_\infty$ is the generator of a strongly continuous contraction semigroup $\{T^{q,t;\alde}_\infty(s): s\ge0\}$ on $C(\wt\Om(q,t;\al,\be))$, which preserves the cone of nonnegative functions and hence gives rise to a conservative Markov process on the space $\wt\Om(q,t;\al,\be)$. The measure $\MM^{q,t;\alde}$ is a unique stationary measure of that process. 
\end{theorem}

This result provides a model of Markov dynamics for an infinite particle system with long-range interaction. A number of models of this kind arise, for instance, in random matrix theory, and they have been intensively studied.  However, unlike many works on infinite-dimensional Markov dynamics, in the present paper neither Dirichlet forms nor stochastic equations are used. Our approach is essentially algebraic and relies on the method of intertwining Markov kernels proposed in \cite{BO-2012} (this method was also used in  \cite{BO-2013} and in the papers by Assiotis \cite{Assiotis}, Borodin-Gorin \cite{BG},  and Cuenca \cite{Cuenca}). Note also that since our limit process lives on a totally disconnected space, the trajectories cannot be continuous, so that the process in not a diffusion.

\subsection{Organization of the paper}

The short section \ref{sect2} contains a list of well-known facts about Feller semigroups that we will need later (generators, pre-generators, the positive maximum principle). In general, Feller semigroups are defined on real Banach spaces of the form $C(E)$, where $E$ is a locally compact space, but we need only compact spaces $E$, which simplifies the picture. 

In section \ref{sect3} we prove the existence of a Feller semigroup whose pre-generator is a second order ordinary $q$-difference operator $D$ related to the univariate big $q$-Jacobi polynomials (Theorem \ref{thm3.A}). Here the compact space $E$ is a $q$-analog of the interval $[-1,1]$: a bounded subset $\wt I_{q;a,b}\subset\R$ formed by $0$ and two geometric progressions with ratio $q$ that converge to $0$ from the both sides:
\begin{equation}
\wt I_{q;a,b}:=\{b^{-1}q,\, b^{-1}q^2,\, b^{-1}q^3,\, \dots, 0,\dots,\, a^{-1}q^3,\, a^{-1}q^2,\, a^{-1}q\},
\end{equation}
where $b<0<a$ are parameters. The operator $D$ depends on $q,a,b$ and two additional parameters $c,d$ such that $c=\bar d\in\C\setminus\R$. We show that $D$ satisfies the positive maximum principle, which readily implies that its closure generates a Feller semigroup on $C(\wt I_{q;a,b})$. The proof in this case is easy; it serves us as an introduction to the more elaborate computation given in section \ref{sect5}. 

In section \ref{sect4} we introduce compact spaces $\wt\Om_N(q,t;a,b)$ and $\wt\Om(q,t;a,b)$. These are, respectively, $N$-dimensional and infinite-dimensional analogs of the $q$-grid $\wt I_{q;a,b}$. Here the additional parameter $t\in(0,1)$ cames into play. In the special case $q=t$, the elements of these spaces can be viewed as particle configurations on $\wt I_{q;a,b}$, but for $q\ne t$ the definition is a bit more sophisticated. The algebra $\Sym$ is realized as a dense subalgebra of $C(\wt\Om(q,t;a,b))$. Likewise, the algebra $\Sym(N)$ of $N$-variate symmetric polynomials is embedded into $C(\wt\Om_N(q,t;a,b))$. 

In section \ref{sect5} we are working with an operator $D_N$ on $\Sym(N)$ which is related to $N$-variate symmetric big $q$-Jacobi polynomials with parameters $(q,t;a,b;c,d)$. The operator $D_N$ is taken from the papers \cite{S}, \cite{SK} (only our notation of the parameters is different, see \eqref{eq5.C}). We regard $D_N$ as an operator on $C(\wt\Om_N(q,t;a,b))$ with the domain $\Sym(N)\subset C(\wt\Om_N(q,t;a,b))$.  Theorem \ref{thm5.A} asserts that in this realization, $D_N$ gives rise to a Feller semigroup $\{T_N(s)\}$ on $C(\wt\Om_N(q,t;a,b))$, which implies in turn that we get a Markov process on $\wt\Om(q,t;a,b)$. This is a $q$-analog of results of Demni \cite{D} and Remling--R\"osler \cite{RR} about the Markov dynamics related to Heckman--Opdam's Jacobi polynomials. Theorem \ref{thm5.A} is proved by a direct verification of the classical positive maximum principle (see its adapted version in section \ref{sect2.4}). The argument is quite elementary but long; it requires a careful analysis of the possible behavior of symmetric polynomials, viewed as functions on $\wt\Om_N(q,t;a,b)$, near extremal points. 

Section \ref{sect6} is a preparation to the last section \ref{sect7}. We list necessary facts about the $N$-variate big $q$-Jacobi polynomials and their analogs in the algebra of symmetric functions. The material is taken from \cite{S}, \cite{SK}, and \cite{Ols-2021b}.

In section \ref{sect7} we prove the main result, Theorem \ref{thm7.A}. It is obtained from Theorem \ref{thm5.A} by applying the method of intertwining Markov kernels from \cite{BO-2012}.

\section{Generalities on Feller semigroups}\label{sect2}

The results collected in this section are contained in a number of  textbooks. Our base reference is Ethier-Kurtz \cite{EK}; some other possible sources are  Ito \cite{Ito}, Kallenberg \cite{Kall}, and Liggett \cite{Ligg}.

Below we denote the time parameter by $s$ instead of the usual $t$ to avoid confusion with the standard notation $(q,t)$ for the parameters of Macdonald's polynomials.

\subsection{Contraction semigroups and their generators}\label{sect2.1}

Let $\F$ be a separable real Banach space. A \emph{contraction semigroup} on $\F$ is a one-parameter strongly continuous semigroup $\{T(s): s\ge0\}$ of operators $\F\to\F$ such that $T(0)=1$ and $\Vert T(s)\Vert\le 1$ for all $s\ge0$. 

The \emph{generator} of $\{T(t)\}$ is the linear operator $A$ on $\F$ defined by
$$
Af:=\lim_{s\to +0} \frac 1s(T(s) f - f).
$$
The \emph{domain} of $A$, denoted by $\Dom (A)$, is the set of all vectors $f\in\F$ for which the above limit exists in the sense of the norm convergence. 

An operator $A$ on $\F$ is the generator of a contraction semigroup $\{T(s)\}$ if and only if $\Dom(A)$ is dense in $\F$, $A$ is \emph{dissipative} (meaning that $\Vert r f- Af\Vert \ge r \Vert f\Vert$ for all $f\in\F$ and $r>0$), and $\Ran (r -A)=\F$ for all $r>0$, where $\Ran(\cdot)$ denotes the \emph{range} of an operator. This is the Hille-Yosida theorem (\cite{EK}, ch. 1, Thm. 2.6). 

A contraction semigroup $\{T(s)\}$ is determined by its generator $A$ uniquely. A way to reconstruct $\{T(s)\}$ from $A$ is as follows (\cite{EK}, ch. 1, Proposition 2.7): for every $f\in\F$,
\begin{equation}\label{eq2.A}
T(s)f=\lim_{r\to+\infty} \exp (s A_r) f, \qquad  A_r:=\frac{r A}{r-A}.
\end{equation}
Note that $A_r$ is a bounded operator, so that $\exp(s A_r)$ is well defined. We use \eqref{eq2.A} in the proof of Theorem \ref{thm7.A}. 

\subsection{Pre-generators}\label{sect2.2}

An operator $A$ on $\F$ is said to be a \emph{pre-generator} of a contraction semigroup $\{T(s)\}$ if $A$ is closable and its closure $\ov A$ is the generator of $\{T(s)\}$. (Note that every generator is closed.) 

$A$ is a pre-generator if and only if it is densely defined, dissipative, and $\Ran(r-A)$ is dense for every $r>0$ (\cite{EK}, ch. 1, Theorem 2.12).

\subsection{Feller semigroups}\label{sect2.3}

Let $E$ is a \emph{compact} metrizable separable space. We take as $\F$ the space $C(E)$ of continuous real-valued functions on $E$ with the supremum norm. 

A contraction semigroup $\{T(s)\}$ on $C(E)$ is said to be a \emph{Feller semigroup} if it is \emph{positive} (meaning that $f\ge 0$ entails $T(s) f\ge0$) and \emph{conservative} (meaning that $T(s)1=1$, where $1$ stands for the constant function equal to $1$; here the conditions should hold for every $s\ge0$).  

This is a simplified version of a more general definition (\cite[ch. 4, section 2.2]{EK}) applicable to locally compact spaces. However, for our purposes, the simpler case of compact spaces is sufficient.

\subsection{Pre-generators of Feller semigroups}\label{sect2.4}
Let $E$ be a compact space, as above. An operator $A$ on $C(E)$ serves as a pre-generator of a Feller semigroup if and only if the following three conditions are satisfied:

(i) $\Dom(A)$ is dense.

(ii)  $\Ran(r-A)$ is dense for all $r>0$. 

(iii) For any function $f\in\Dom(A)$ and any point $x\in E$ where $f$ attains its minimum, one has $Af(x)\ge0$.

This is a slight adaptation, to the case of compact spaces, of a more general result (see, e.g. \cite[ch. 4, Theorem 2.2]{EK}). The condition (iii) is a simplified version of the  \emph{positive maximum principle}. We prefer a formulation with points of minimum, not maximum; there is no difference as $f$ can be replaced with $-f$. 

\subsection{Feller processes}\label{sect2.5}

Let $E$ be a compact space, as above. Every Feller semigroup $\{T(s)\}$ on $C(E)$ gives rise to a continuous time Markov process $X(s)$ on $E$  such that 
$$
T(s)f(x)=\int_E p(s;x,dy)f(y), \qquad f\in C(E), \quad x\in E, \quad s\ge0,
$$
where $p(s;x,dy)$ denotes the transition function of $X(s)$. Moreover, the sample trajectories of $X(s)$ are almost surely \emph{c{\`a}dl{\`a}g} (meaning that they are right-continuous and have left limits).

See \cite[ch. 4, Theorem 2.7]{EK} for a stronger formulation.  
This theorem provides a convenient tool for constructing Markov processes with good properties. In the present paper, however, we do not discuss probabilistic questions and focus on Feller semigroups only.

\section{One-particle processes}\label{sect3}

\subsection{Some notation}

Let $a,b,c,d$, and $q$ be fixed parameters. At this moment they can be complex numbers$, abq\ne0$; further constraints will be imposed shortly. 

We denote by  $\si^+(x)$ and $\si^-(x)$ two rational functions of a variable $x$, given by
\begin{gather}
\si^+(x)=-\,\frac{q}{ab} (c-x^{-1})(d-x^{-1}), \label{eq3.C1}\\
\si^-(x)= -\, \frac{q^2}{ab}\,(aq^{-1}-x^{-1})(bq^{-1}-x^{-1}). \label{eq3.C2}
\end{gather}

The \emph{$q$-shift operator} $S_q^+$ and its inverse $S_q^-$ are defined by 
$$
S_q^\pm f(x)=f(xq^{\pm1}),
$$
where $f(x)$ is a test function. (The standard notation for the $q$-shift operator is $T_q$, but we replace ``$T$'' by ``$S$'' to avoid confusion with the notation for Feller semigroups.)

\subsection{The ordinary $q$-difference operator $D$}

\begin{definition}[cf. Groenvelt \cite{G-2009}, sect. 2.1]
We set 
$$
D:=\si^+(x)(S_q^+-1)+\si^-(x)(S_q^--1).
$$
In more detail, $D$ acts on a test function $f(x)$ by
\begin{equation}\label{eq3.B}
Df(x)=\si^+(x)(f(xq)-f(x))+\si^-(x)(f(xq^{-1}-f(x)),
\end{equation}
where we write $Df(x)$ instead of the more pedantic notation $(D f)(x)$.
\end{definition}

\begin{lemma} 
$D$ acts on the monomials $x^n$, $n\in\Z$, according to the following formula
\begin{multline}\label{eq3.A}
D x^n=-(q^{-n}-1)\left(\frac{-\, cdq^{n+1}}{ab}+1\right) x^n\\
+\frac{q}{ab}\left((c+d)(q^n-1)+(a+b)(q^{-n}-1)\right)x^{n-1}\\
-\frac{q}{ab}\left((q^n-1)+q(q^{-n}-1)\right) x^{n-2}.
\end{multline}
\end{lemma}

\begin{proof}
Direct computation.
\end{proof}

\begin{corollary}
$D$ is correctly defined on $\C[x]$, the space of polynomials.
\end{corollary}

\begin{proof} 
From \eqref{eq3.A} it is clear that $D$ acts on the space $\C[x,x^{-1}]$ of Laurent polynomials. To prove that $D$ preserves $\C[x]$ we will check that $D x^n\in\C[x]$ for all $n\in\Z_{\ge0}$. Since $D$ decreases degree at most by $2$, the only cases to be checked are that of $n=0$ and $n=1$. This is easy to do. Indeed:

$\bullet$ for $n=0$, all coefficients in \eqref{eq3.A} vanish because $q^{\pm n}-1=0$ for $n=0$; this is also seen from the initial definition of $D$, because $S_q^\pm-1$ annihilates the constants; 

$\bullet$ for $n=1$, the coefficient of $x^{n-2}=x^{-1}$ vanishes, so that $D x$ does not have a nonzero term proportional to $x^{-1}$. 

This completes the proof. 
\end{proof}

\subsection{Constraints on the $5$-tuple $(q; \abcd)$}\label{sect3.3}

We will assume that 
$$
0<q<1, \quad b<0<a, \quad  c\in\C\setminus\R,  \quad d=\ov c. 
$$

These constraints imply in particular that $D$ is well defined on $\R[x]$. 

\subsection{The $q$-lattice $\Iqab$ and its compactification $\wt I_{q;a,b}$}

We set 
$$
\begin{aligned}
\Iqab:&=a^{-1}q^{\Z_{\ge1}}\cup b^{-1}q^{\Z_{\ge1}}\\
&=\{b^{-1}q, b^{-1}q^2, b^{-1}q^3,\dots\}\cup \{\dots, a^{-1}q^3, a^{-1}q^2,  a^{-1}q\}.
\end{aligned}
$$
This is a bounded countable subset of $\R\setminus\{0\}$ which has $0$ as its (unique) accumulation point. Next, we set 
$$
\wt I_{q;a,b}:=\Iqab\cup\{0\}.
$$
It is a compact subset of $\R$. 

Note that polynomials in $x$ are uniquely determined by their restriction to $\wt I_{q;a,b}$. This enables us to regard $\R[x]$ as a subspace of the Banach space $C(\wt I_{q;a,b})$. By the Stone-Weierstrass theorem, this subspace is dense. 

In what follows we regard $D$ as an operator on $C(\wt I_{q;a,b})$ with domain $\R[x]$. 

Due to the constraints imposed on the parameters we have  
\begin{gather}
\text{$\si^+(x)>0$ for all $x\in\wt I_{q;a,b}$}, \label{eq3.E1} \\
\text{$\si^-(x)>0$ for all $x\in I_{q;a,b}$},   \label{eq3.E2}\\
\si^-(a^{-1}q)=\si^-(b^{-1}q)=0. \label{eq3.E3}
\end{gather}
This will be important for us.

\subsection{Markov dynamics on $\wt I_{q;a,b}$}\label{sect3.5}

Recall that the notions of Feller semigroup and its pre-generator were defined in section \ref{sect2}.

\begin{theorem}\label{thm3.A}
The operator $D$ serves as a pre-generator of a Feller semigroup on $C(\wt I_{q;a,b})$.
\end{theorem}

Combining this result with the general theorem from subsection \ref{sect2.5} we obtain that $D$ gives rise to a Markov process on $\wt I_{q;a,b}$.

\begin{proof}
We apply the criterion stated in section \ref{sect2.4}. 

Condition (i) is obvious. Indeed, $\Dom(D)$ is $\R[x]$, which is dense in $C(\wt I_{q;a,b})$. 

Condition (ii) is checked as follows. From \eqref{eq3.A} and the constraints on the parameters it follows that 
$$
D x^n=\m_n x^n + \text{lower degree terms}, \quad n\in\Z_{\ge0},
$$
where $\m_0=0$ and $\m_n<0$ for $n\ge1$. From this we see that for every $r>0$, the matrix of the operator $r-D$  in the basis $\{x^n: n\in\Z_{\ge0}\}$ has triangular form, with nonzero entries on the diagonal. Therefore, $\Ran(r-D)$ is the whole space $\R[x]$, as required. 

We turn to condition (iii). Let $f\in\R[x]$ be arbitrary and $\wt x\in\wt I_{q;a,b}$ be a point where $f$ attains its minimum on $\wt x\in\wt I_{q;a,b}$. We have to check that $D f(\wt x)\ge0$. Let us examine possible variants for the location of the point  $\wt x$.  

\smallskip

\emph{Variant} \/1: $\wt x$ is neither an endpoint, nor $0$. This means that both $\wt x\in\Iqab$ and $\wt xq^{\pm1}\in\Iqab$. Then it follows from \eqref{eq3.E1} and \eqref{eq3.E2} that $\si^\pm(\wt x)>0$. On the other hand, since $\wt x$ is a point of minimum, $f(\wt x q^{\pm1})-f(\wt x)\ge0$. Then \eqref{eq3.B} shows that $D f(\wt x)\ge0$.

\smallskip

\emph{Variant} \/2: $\wt x$ is an endpoint, that is, $\wt x=b^{-1}q$ or $\wt x=a^{-1}q$. Then $\wt x q^{-1}$ is outside $\wt I_{q;a,b}$, so we cannot  assert anymore that $f(\wt x q^{-1})-f(\wt x)\ge0$. However, we do not need this, because the factor $\si^-(\wt x)$ vanishes due to \eqref{eq3.E3}, so that the summand with $f(\wt x q^{-1})-f(\wt x)$ disappears. We are left with the remaining summand and the previous argument works.  

\smallskip

\emph{Variant} \/3: $\wt x=0$. It suffices to prove that if $0$ is a point of local minimum, then $D^{q;\abcd} f(0)\ge0$. Write $f(x)$ as a linear combination of monomials: 
$$
f(x)=f_0+f_1x +f_2 x^2+\dots\,.
$$
The constant term $f_0$ plays no role, because $D$ annihilates the constants. Thus, we may assume $f_0=0$, so that $f(x)=f_1 x+O(x^2)$ near $x=0$. If $f_1\ne0$, then $f(x)$ changes the sign near $0$, which cannot happen, as $0$ is a point of minimum. Therefore, $f_1=0$ and hence $f(x)=f_2 x^2+ O(x^3)$ near $0$. Again, since $0$ is a point of minimum, we have $f_2\ge0$. 

Next, we apply formula \eqref{eq3.A}. Observe that $D f(0)=f_2 \cdot(Dx^2)(0)$, because $D$ decreases degree at most by $2$. It suffices to check that $(D x^2)(0)>0$. 

Setting $n=2$ in \eqref{eq3.A} we obtain
$$
(D x^2)(0)=-\frac{q}{ab}\left((q^2-1)+q(q^{-2}-1)\right)=-\, \frac{(1-q)(1-q^2)}{ab}.
$$
Since $ab<0$ and $0<q<1$, this quantity is positive. 

This completes the proof. 
\end{proof}

\section{The spaces $\wt\Om$ and $\wt\Om(q,t;a,b)$}\label{sect4}

\subsection{The space $\wt\Om$}

Let $\N:=\{1,2,3,\dots\}$ be the set of positive integers and $\NN:=\N\cup\{\infty\}$ be its one-point compactification. 
We endow $\NN$ with the natural total order, so that $n<\infty$ for any $n\in\N$. Let 
$\Ninfty:=\NN\times\NN\times\dots$ be the the direct product of countably many copies of $\NN$. It is a compact space with respect to the product topology. Elements of $\Ninfty$ are infinite sequences $(n_1,n_2,\dots)$ where $n_i\in\NN$ for each $i$.  We say that $n_1,n_2,\dots$ are the \emph{coordinates} of the sequence. A sequence is said to be \emph{monotone} if its coordinates are weakly increasing: $n_1\le n_2\le \dots$\,. 

\begin{definition}\label{def4.A}
Let $\wt\Om\subset \Ninfty\times\Ninfty$ be the subset  of pairs $\om=(\om^+,\om^-)$ such that both sequences  $\om^+\in\Ninfty$ and $\om^-\in\Ninfty$ are monotone.
\end{definition} 

Clearly $\wt\Om$ is closed in $\Ninfty\times\Ninfty$ and hence is also a compact space. It is metrisable and totally disconnected. 

\begin{definition}\label{def4.B}
(i) For $\om=(\om^+,\om^-)\in\wt\Om$, we set 
$$
|\om|:=|\om^+|+|\om^-|,
$$
where $|\om^\pm|$ stands for the number of coordinates in $\om^\pm$ not equal to $\infty$; this number may be equal to infinity; so $|\om|$ may be equal to infinity, too, because we agree that $\infty+\infty=\infty$. 

(ii) For $N=1,2,\dots$ we set
$$
\Om_N:=\{\om\in\wt\Om: |\om|=N\}, \quad \wt\Om_N:=\{\om\in\wt\Om: |\om|\le N\}.
$$
We also agree that $\Om_0=\wt\Om_0$ is the singleton formed by the unique element $\om$ with $|\om|=0$ (all coordinates are equal to $\infty$).

(iii) We set
$$
\Om_\infty:=\{\om\in\wt\Om: |\om|=\infty\}.
$$
\end{definition}

The following assertions are evident:

\sms

$\bullet$ $\wt\Om$ and $\wt\Om_N$ can be represented as disjoint unions of subsets,
$$
\wt\Om=\Om_0\sqcup\Om_1\sqcup\Om_2\sqcup\dots\sqcup\Om_\infty, \qquad \wt\Om_N=\Om_0\sqcup\dots\sqcup \Om_N.
$$

\sms

$\bullet$ For each $N$, the set $\wt\Om_N$ is the closure of $\Om_N$, so that the sets $\wt\Om_N$ form an increasing chain of compact subsets of $\wt\Om$. 

\sms

$\bullet$ Their union is a countable dense subset of $\wt\Om$. 

\sms

$\bullet$ Its complement, $\Om_\infty$, is a dense subset of $\wt\Om$, too. Note that it is uncountable.

\subsection{The bijection $\wt\Om\to\wt\Om(q,t;a,b)$}\label{sect4.2}

We fix the parameters $q,t,a,b$ such that $q,t\in(0,1)$ and $b<0<a$. To each $\om=(\om^+,\om^-)\in\wt\Om$ we assign a pair $X=(X^+,X^-)$ of sequences of real numbers as follows. Let $(k_1,k_2,\dots)$ and $(l_1,l_2,\dots)$ be the coordinates of $\om^+$ and $\om^-$, respectively. Then we set 
\begin{equation}\label{eq4.B}
\begin{aligned}
\text{$X^+:=(x^+_1,x^+_2,\dots)$, where $x^+_i:=a^{-1}q^{k_i}t^{i-1}$ for $i=1,2,\dots$} \\
\text{$X^-:=(x^-_1,x^-_2,\dots)$, where $x^-_i:=b^{-1}q^{l_i}t^{i-1}$ for $i=1,2,\dots$},
\end{aligned}
\end{equation}
with the understanding that $q^\infty:=0$. 

We write $X=X(\om)$ and $X^\pm=X^\pm(\om^\pm)$. The correspondence $\om\mapsto X(\om)$ depends on the quadruple $(q,t;a,b)$ but to simplify the notation, these parameters are suppressed here.  

The nonzero coordinates of $X^+$ are positive and strictly decreasing, while the nonzero coordinates of $X^-$ are negative and strictly increasing. More precisely, the nonzero coordinates satisfy the inequalities
\begin{equation}\label{eq4.A}
\frac{x^\pm_{i+1}}{x^\pm_i} \le t
\end{equation}
so that if there are infinitely many nonzero coordinates, then they converge to $0$ at least as quickly as a geometric progression with ratio $t$. 

All coordinates $x^\pm_i$ are contained in the closed interval $[b^{-1}q, a^{-1}q]$ surrounding the zero. 

The correspondence $\om\mapsto X(\om)$ is evidently injective. We denote by $\wt\Om(q,t;a,b)$ the image of $\wt\Om$ under this map and we equip $\wt\Om(q,t;a,b)$ with the topology of $\wt\Om$. Thus, $\wt\Om(q,t;a,b)$ becomes a compact space.

We also denote by $\Om_N(q,t;a,b)$, $\wt\Om_N(q,t;a, b)$, and $\Om_\infty(q,t;a,b)$ the images of $\Om_N$, $\wt\Om_N$, and $\Om_\infty$, respectively. 

Evidently, $\Om_\infty(q,t;a,b)$ is dense in $\wt\Om_\infty(q,t;a,b)$, $\Om_N(q,t;a, b)$ is dense in $\wt\Om_N(q,t;a, b)$, and $\wt\Om_N(q,t;a, b)$ is compact. 

Ignoring possible zero coordinates we may regard elements $X\in\wt\Om(q,t;a,b)$ as particle configurations on the punctured real line $\R^*$, so that $\wt\Om(q,t;a,b)$ becomes a subset of $\Conf(\R^*)$, in the notation of subsection \ref{sect1.1}. In this interpretation, the elements from $\Om_\infty(q,t;a,b)$ or $\Om_N(q,t;a,b)$ are configurations with infinitely many particles or $N$ particles, respectively.  

Alternatively, the elements $X\in\wt\Om_N(q,t;a,b)$ can be treated as $N$-particle configurations, provided that particles are allowed to stick together at $0$. In other words, then $X\subset\R$ becomes a \emph{multiset} of power $|X|=N$, with possible multiple points at $\{0\}$. This interpretation is used in section \ref{sect5}. 

In the case $N=1$ we return to the context of section \ref{sect3}: $\Om_1(q,t;a,b)=I_{q;a,b}$ and $\wt\Om_1(q,t;a,b)=\wt I_{q;a,b}$; the dependence of $t$ disappears here.

\subsection{The embedding $\Sym\hookrightarrow C(\wt\Om(q,t;a,b))$}

Recall that $\Sym$ denotes the algebra of symmetric functions over $\R$. To each $f\in\Sym$, we assign a function $f(\om)$ on the space $\wt\Om(q,t;a,b)$ by setting 
$$
f(\om):=f(X(\om)), \quad \om\in\wt\Om,
$$
meaning that we evaluate $f$ at the double collection $\{x^\pm_i;i=1,2,\dots\}$ of the coordinates of $X(\om)$. For the generators $p_1,p_2,p_3,\dots$ of $\Sym$ (the Newton power sums) this means 
\begin{equation}\label{eq4.C}
p_m(\om)=\sum_{i=1}^\infty (x^+_i)^m + \sum_{i=1}^\infty (x^-_i)^m= \sum_{i=1}^\infty (x^+_i)^m + (-1)^m\sum_{i=1}^\infty |x^-_i|^m, \quad m=1,2,3,\dots,
\end{equation}
and from \eqref{eq4.A} it follows that these sums converge. Thus, our definition for makes sense for the elements $p_m$ and hence for all $f\in\Sym$ as well.

Let $C(\wt\Om(q,t;a,b))$ be the commutative real Banach algebra of continuous functions on the compact space $\wt\Om(q,t;a,b)$ with pointwise operations and the supremum norm.

The next lemma will be used in sections \ref{sect6} and \ref{sect7}. 

\begin{lemma}\label{lemma4.A}

{\rm(i)} The functions on $\wt\Om(q,t;a,b)$ coming from elements $f\in\Sym$ are continuous, so that we obtain a unital algebra  morphism $\Sym\to C(\wt\Om(q,t;a,b))$.

{\rm(ii)} This morphism is injective.

{\rm(iii)} Its image is dense.

\end{lemma}

\begin{proof} 
(i) Because $x^+_1$ and $|x^-_1|$ are uniformly bounded (by the constants $a^{-1}q$ and $|b^{-1}|q$, respectively), the sums in \eqref{eq4.C} converge uniformly in $\om\in\wt\Om(q,t;a,b)$. Finally, from \eqref{eq4.B} and the definition of the topology on $\NN$ it follows that each coordinate $x^\pm_i$ is a continuous function on $\wt\Om(q,t;a,b)$. This proves (i).

(ii) Consider another system of generators of $\Sym$, the elementary symmetric functions $e_1,e_2,\dots$, and their standard generating series $E(z):=1+\sum_{m=1}^\infty e_m z^m$. Its specialization at $X(\om)$ is
$$
E(\om;z):=1+\sum_{m=1}^\infty e_m(X(\om))=\prod_{i=1}^\infty (1+x^+_i z)\cdot \prod_{i=1}^\infty (1+x^-_i z).
$$
This is an entire function of $z$, and its zeroes determine the nonzero coordinates of $X(\om)$ uniquely. This proves (ii). 

(iii) We apply the real version of the Stone-Weierstrass theorem. Because $\Sym$ contains $1$, it suffices to check that the functions from $\Sym$ separate the elements of $\wt\Om(q,t;a,b)$. But this is clear from the form of $E(\om;z)$: if $\om$ and $\om'$ are two distinct elements, then the corresponding entire functions $E(\om;z)$ and $E(\om';z)$ are distinct, too. This completes the proof. 
\end{proof}

\subsection{The embedding $\Sym(N)\hookrightarrow C(\wt\Om_N(q,t;a,b))$}

Let $N$ be a fixed positive integer. The algebra $\Sym(N)$ of symmetric polynomials in $N$ variables can be identified with the quotient of $\Sym$ by the ideal generated by the elements $e_n$ with $n>N$. Lemma \ref{lemma4.A} allows us to identify $\Sym$ with a subalgebra of $C(\wt\Om(q,t;a,b))$. The restriction map $f\mapsto f\big |_{\wt\Om_N(q,t;a,b)}$ annihilates all $e_n$ with indices $n>N$ and hence gives rise to a unital  algebra morphism $\Sym(N)\to C(\wt\Om_N(q,t;a,b))$. 

\begin{lemma}\label{lemma4.B}
The morphism $\Sym(N)\to C(\wt\Om_N(q,t;a,b))$ just defined is injective with dense image.
\end{lemma}

\begin{proof} 
The same argument as in Lemma \ref{lemma4.A}, with the difference that we use only the first $N$ generators $e_1,\dots,e_N$. 
\end{proof}

\section{$N$-particle processes}\label{sect5}

\subsection{Preliminaries}\label{sect5.1}

We still assume that $q,t\in(0,1)$, $b<0<a$, $c=\bar d\in\C\setminus\R$. We fix $N\in\{2,3,\dots\}$. The functions $\si^+_N(x)$ and $\si^-_N(x)$ are given by (cf. \eqref{eq3.C1}, \eqref{eq3.C2})
\begin{gather}
\si^+_N(x)=t^{N-1}\si^+(x)=-\,\frac{qt^{N-1}}{ab} (c-x^{-1})(d-x^{-1}),   \label{eq5.A1} \\
\si^-_N(x)=t^{N-1}\si^-(x)=- \, \frac{q^2 t^{N-1}}{ab} (aq^{-1}-x^{-1})(bq^{-1}-x^{-1}).  \label{eq5.A2}
\end{gather}  
Given an $N$-tuple of variables $X=(x_1,\dots,x_N)$, we set 
$$
V_N=V_N(X)=\prod_{1\le i<j\le N}(x_i-x_j).
$$
Together with the $q$-shift operators $S^{\pm}_q$ we use similar $t$-shift operators: their action on a test function $f(x)$ is given by
$$
S_t^\pm(x)=f(xt^{\pm1}).
$$
We denote by $S_{q,i}^\pm$ and $S_{t,i}^\pm$ the operators $S_q^\pm$ and $S_t^\pm$ acting on the $i$th variable $x_i$, where $1\le i\le N$. 

Using this notation we introduce a partial $q$-difference operator in $N$ variables:
\begin{equation}\label{eq5.B}
D_N:=\sum_{i=1}^N\left(\frac{S_{t,i}^+ V_N}{V_N}\si_N^+(x_i)(S_{q,i}^+-1)+\frac{S_{t,i}^- V_N}{V_N}\si^-_N(x_i)(S_{q,i}^--1)\right).
\end{equation}
This definition is taken from Stokman \cite{S}. Specifically, the correspondence between our parameters $a,b,c,d$ and the parameters in \cite{S}, which we rename to $A,B,C,D$ in order to avoid confusion, is the following: 
\begin{equation}\label{eq5.C}
\begin{gathered}
A=ca^{-1}, \; B=db^{-1}, \; C=a^{-1}q,\; D=-b^{-1}q, \\
a=q C^{-1}, \; b=-q D^{-1}, \; c=Aq C^{-1}, \; d=-Bq D^{-1}.
\end{gathered}
\end{equation}
With this convention, $D_N$ coincides with the operator $- D^{A,B,C,D}_{N,q,t}$ from  \cite[(3.8)]{S} (note the minus sign!). 

\begin{lemma}\label{lemma5.A}
The operator $D_N$ is correctly defined on the space $\Sym(N)$.
\end{lemma}

\begin{proof}
From \cite[Lemma 4.1]{S} it follows that $D_N$ is defined on the space of symmetric polynomials over $\C$. But the same holds over the base field $\R$, because the coefficients $\si^\pm_N(x)$ take real values on $\R$ due to our conditions on the parameters. 
\end{proof}

\subsection{Verification of positive maximum principle for the operator $D_N$}

Using Lemma \ref{lemma4.B} we realize $\Sym(N)$ as a dense subspace of $C(\wt\Om_N(q,t;a,b))$.

The next proposition generalizes the main part of Theorem \ref{thm3.A}. 

\begin{proposition}\label{prop5.A}
Let $f\in\Sym(N)$ be arbitrary and $\wt X\in\wt\Om_N(q,t;a,b)$ be a configuration where $f$ attains its minimum on\/ $\wt\Om_N(q,t;a,b)$. Then $D_Nf(\wt X)\ge0$.
\end{proposition}

The proof is divided into a series of lemmas. It is convenient for us to treat elements of $\wt\Om_N(q,t;a,b)$ as multisets (see the end of subsection \ref{sect4.2}) and then enumerate their points in some way; for instance, in the descending order. This allows us to  regard $\wt\Om_N(q,t;a,b)$ as a subset of $\R^N$. 

Every $f\in\Sym(N)$ can be expanded into homogeneous components of degrees $0,1,2,\dots$:
$$
f=f_0+f_1+f_2+\dots\,.
$$
Setting 
$$
p_{1,N}:=\sum_{i=1}^N x_i, \qquad p_{2,N}:=\sum_{i=1}^N x_i^2, \qquad e_{2,N}:=\sum_{1\le i<j\le N}x_ix_j,
$$
we can write
$$
f_0=c_0, \qquad f_1=c_1 p_{1,N}, \qquad f_2=c'_2 p_{2,N}+ c''_2 e_{2,N},
$$
where $c_0, c_1, c'_2, c''_2$ are some real constants. 

Observe that if $f$, viewed as a function on $\R^N$, has a local minimum at the origin, then $c_1=0$ and the quadratic form $f_2$ is nonnegatively definite. As is readily checked, the latter condition means that either $f_2=0$ (that is, $c'_2=c''_2=0$) or
\begin{equation}\label{eq5.D}
c'_2>0, \qquad -2\le \dfrac{c''_2}{c'_2}\le 2.
\end{equation}
 
However, we have to treat $f$ as a function on the subset $\wt\Om_N(q,t;a,b)\subset\R^N$, and then the condition on $f_2$ changes. The reason is  that the set $\wt\Om_N(q,t;a,b)$ is sparse; because of this it may well happen that a quadratic form, which is  non-degenerate and not positive definite, has a local minimum at the origin as a function on $\wt\Om_N(q,t;a,b)$.  

\begin{lemma}\label{lemma5.B}
Assume $f\in\Sym(N)$ has a local minimum on $\wt\Om_N(q,t;a,b)$ at the point 
$$
0^N:=(0,\dots,0) \quad \text{{\rm(}$N$ times{\rm)}}.
$$
Then one has:

{\rm(i)} $c_1=0$,

{\rm(ii)} either $f_2=0$ or the coefficients $c'_2, c''_2$ satisfy the inequalities
$$
c'_2>0, \qquad  - \, \max\left(\dfrac{1+q}{\sqrt q}, \, \dfrac{1+t^2}{t}\right)  \le \dfrac{c''_2}{c'_2}\le \dfrac{1+q}{\sqrt q}.
$$
\end{lemma}

Note that $(1+q)/{\sqrt q}>2$ because $0<q<1$. Thus, we get a weaker condition  than in \eqref{eq5.D}. For our purposes, we 
actually need only the upper bound on $c''_2/c'_1$.

\begin{proof}
Let us examine the case $N=2$; later it will be clear that the same argument works for $N\ge3$ as well. Thus, we write  $f=f(x_1,x_2)$. First of all, without lost of generality,  we may assume $c_0=0$, so that $f(0,0)=0$. 

Claim (i) is easy. Indeed, recall that according to our convention $x_1\ge x_2$. We have 
$$
f(x_1,0)=c_1 x_1+ O(|x_1|^2), \qquad f(0,x_2)=c_1 x_2+O(|x_2|^2).
$$
It follows that if $c_1<0$, then $f(x_1,0)$ takes negative values for small $x_1>0$, while if $c_1>0$, then $f(0,x_2)$ takes negative values for small $x_2<0$. In both cases, there is no local local minimum at $(0,0)$, hence $c_1=0$. 

We proceed to claim (ii). Now we have
$$
f(x_1,x_2)=c'_2(x_1^2+x_2^2)+c''_2x_1x_2 +O((|x_1|+|x_2|)^3).
$$
If both coefficients equal $0$, there is nothing to prove. Next, it is easy to exclude the cases $c'_2=0$ and $c'_2<0$. Indeed,  if $c'_2=0$, then taking both variables small and such that the sign of their product is opposite to that of $c''_2$ leads to contradiction. Next, if $c'_2<0$, then a contradiction arises when one variable is small and the other vanishes (so that the product $x_1x_2$ disappears).  Therefore,  $c'_2>0$.

Now, without loss of generality, we may assume $c'_2=1$. To simplify the notation, we also rename $c''_2=u$. We are going to prove the upper bound $u\le(1+q)/\sqrt q$. 

To do this we will examine the behavior of $f(x_1,x_2)$ near $(0,0)$ assuming $x_1>0>x_2$. Setting $y:=x_2/x_1$, we obtain
$$
f(x_1,x_2)=x_1^2(1+y^2+uy)+O(x_1^3(1+|y|)^3).
$$

The definition of $\wt\Om_2(q,t;a,b)$ imposes the following constraints on $x_1$ and $x_2$: 
$$
x_1\in\{a^{-1} q^m: m=1,2,\dots\}\subset\R_{>0}, \qquad x_2\in\{b^{-1}q^n: n=1,2,\dots\}\subset\R_{<0}.
$$
It follows that $y$ can take any value from the $q$-lattice 
$$
Y_-:=\{ab^{-1} q^k: k\in\Z\}\subset\R_{<0}
$$ 
provided $x_1$ is small enough. 

Therefore, it suffices to show that if $u>(1+q)/\sqrt q$, then there exists a point $y\in Y_-$ such that $1+y^2+uy<0$; then this will lead to a contradiction. 

Since $u>2$, the polynomial $1+y^2+uy$  has two real roots $y_1<y_2$ on $\R_{<0}$:
$$
y_1:= -\, \frac{u+\sqrt{u^2-4}}{2}, \qquad y_2:= -\, \frac{u-\sqrt{u^2-4}}{2}.
$$
In the interval $(y_1,y_2)$, the polynomial takes negative values. This interval must contain a point of $Y_-$ as soon $\log (|y_1|/|y_2|)>\log(1/q)$, which means 
$$
\frac{|y_2|}{|y_1|}<q
$$
This inequality can be written as
$$
u-\sqrt{u^2-4}< (u+\sqrt{u^2-4})q
$$
or else as
$$
u(1-q)<\sqrt{u^2-4}\,(1+q),
$$
which is equivalent to the inequality  $u>(1+q)/\sqrt q$.  This completes the derivation of the upper bound.

The lower bound is obtained by the same method. Here we take $x_1>x_2>0$ and use the fact the ratio $y:=x_2/x_1$ ranges over the set 
\begin{equation}\label{eq5.E}
Y_+:=\{q^k t: k=0,1,2,\dots\}.
\end{equation}
Then we have to find a condition on $u<-2$ under which there is a point of $Y_+$ inside the interval $(y_2,y_1)\subset\R_{>0}$ between the two roots, 
where the polynomial $1+y^2+uy$ takes negative values. The roots are  
$$
y_2:=\frac{|u|-\sqrt{u^2-4}}{2}, \qquad y_1:=\frac{|u|+\sqrt{u^2-4}}{2},
$$

Again, the set $Y_+$ must be dense enough, so we again impose the inequality $\log(y_1/y_2)>\log(1/q)$, equivalent to $|u|>(1+q)/\sqrt q$. However, this does not suffice: the reason is that the parameter $k$ in \eqref{eq5.E} ranges over $\Z_{\ge0}$ and not over $\Z$, as before. To guarantee that $Y_+$ intersects the interval we need to impose the additional constraint $y_2< t$, which is equivalent to $|u|>(1+t^2)/t$. In this way one gets the lower bound of the lemma. 

Finally, for $N\ge 3$ the arguments are the same. The only difference is that for the upper bound, we set $x_2=\dots=x_{N-1}=0$ and work with $x_1$ and $x_N$, while for the lower bound, we work with $x_1$ and $x_2$ assuming $x_3=\dots=x_N=0$. 
\end{proof}

In the next lemma we compute the constant terms of the symmetric polynomials $D_N p_{2,N}$ and $D_N e_{2,N}$. Below $\CT(\cdot)$ denotes the constant term of a polynomial.

\begin{lemma}\label{lemma5.C}
Set
\begin{equation}\label{eq5.J}
C_N(q,t):=\frac{q t^{N-1}}{a|b|} \,(1-q)\frac{1-t^N}{1-t}.
\end{equation}
{\rm(i)} The constant term $\CT(D_N p_{2,N})$ is positive and  equals 
\begin{equation}\label{eq5.G1}
C_N(q,t) (1+q)  (t^{1-N}q^{-1}-1).
\end{equation}

{\rm(ii)} The constant term $\CT(D_N e_{2,N})$ is negative and  equals 
\begin{equation}\label{eq5.G2}
-\, C_N(q,t) (t^{1-N}-1).
\end{equation}
\end{lemma}

\begin{proof}
Introduce some notation. Let $\wt D_N$ denote the operator obtained from $D_N$ by replacing the rational functions $\si_N^+(x_i)$ and $\si_N^-(x_i)$ from \eqref{eq5.A1} and \eqref{eq5.A2} with their terms of degree $-2$, which are, respectively,
$$
-\, \frac{q t^{N-1}}{ab}\cdot \frac1{x_i^2}=\frac{q t^{N-1}}{a|b|}\cdot \frac1{x_i^2} \qquad \text{and} \qquad - \frac{q^2t^{N-1}}{ab}\cdot \frac1{x_i^2}=\frac{q^2t^{N-1}}{a|b|}\cdot \frac1{x_i^2}.
$$

Next, let $\de:=(N-1,N-2,\dots,0)$ and
$$
\x^\de:=x_1^{N-1}x_2^{N-2}\dots x_{N-1}
$$
be the leading monomial in $V_N$. 

Finally, let $\mathfrak S_N$ be the $N$th symmetric group and $\Anti_N$ be the antisymmetrization operator with respect to the variables $x_1,\dots,x_N$; it acts on a function $F(x_1,\dots,x_N)$ by 
$$
\Anti_N F(x_1,\dots,x_N):=\sum_{s\in\mathfrak S_N}\sgn(s) F(x_{s(1)},\dots,x_{s(N)}),
$$
so that
$$
V_N=\Anti_N \x^\de.
$$

Given an arbitrary symmetric homogeneous polynomial $\phi(x_1,\dots,x_N)$ of degree $2$, we have
$$
\CT(D_N \phi)=\wt D_N \phi= \frac{q t^{N-1}}{a|b|} \frac1{V_N} \sum_{i=1}^N
\left((S_{t,i}^+ V_N)\frac1{x_i^2}\cdot(S_{q,i}^+-1)\phi
+q(S_{t,i}^- V_N)\frac1{x_i^2}\cdot(S_{q,i}^--1)\phi\right)
$$
Because $\phi$ is symmetric, this is equal to the coefficient of $\x^\de$ in 
$$
\frac{q t^{N-1}}{a|b|} \,\Anti_N\left\{\sum_{i=1}^N
\left((S_{t,i}^+ \x^\de)\frac1{x_i^2}\cdot(S_{q,i}^+-1)\phi
+q(S_{t,i}^- \x^\de)\frac1{x_i^2}\cdot(S_{q,i}^--1)\phi\right)\right\}.
$$
Next, 
$$
S_{t,i}^\pm \x^\de=t^{\pm(N-i)}\x^\de.
$$
Therefore, our expression can be rewritten as
\begin{equation}\label{eq5.F}
\frac{q t^{N-1}}{a|b|} \,\Anti_N\left\{\sum_{i=1}^N
\left(t^{N-i}\x^{\de-2\eps_i}\cdot(S_{q,i}^+-1)\phi
+q t^{i-N}\x^{\de-2\eps_i}\cdot(S_{q,i}^--1)\phi\right)\right\},
\end{equation}
where $\eps_i:=(0,\dots,0,1,0,\dots,0)$ with `$1$' at the $i$th position, so that
$$
\x^{\de-2\eps_i}=\frac{\x^\de}{x_i^2}.
$$

Now we specialize \eqref{eq5.F} to $f=p_{2,N}$. Then 
$$
(S^\pm_{q,i}-1)p_{2,N}=(q^{\pm 2}-1)x_i^2.
$$
Therefore, \eqref{eq5.F} turns into
$$
\frac{q t^{N-1}}{a|b|} \,\Anti_N\left\{\left(\sum_{i=1}^N
\left(t^{N-i}(q^2-1)
+t^{i-N}q(q^{-2}-1)\right) \right)\x^\de\right\},
$$
In this expression, the coefficient of $\x^\de$ equals
$$
\frac{q t^{N-1}}{a|b|} \, \sum_{i=1}^N
\left(t^{N-i}(q^2-1)
+t^{i-N}q(q^{-2}-1)\right).
$$
Summing the two geometric progressions we obtain
$$
\frac{q t^{N-1}}{a|b|} \,(1-q^2)\left(- \frac{1-t^N}{1-t} + q^{-1} \frac{t^{-N}-1}{t^{-1}-1}\right),
$$
which equals \eqref{eq5.G1}, as desired.

Next, we specialize \eqref{eq5.F} to 
$$
\phi=e_{2,N}=\sum_{1\le k<l\le N}x_kx_l.
$$ 
Then the expression in the curly brackets in \eqref{eq5.F} becomes a triple sum over $k<l$ and $i$. However, a nonzero contribution can  arise only from $i=k$ and $i=l$. Indeed, this follows from the fact that $(S^\pm_{q,i}-1)x_kx_l$ vanishes if $i$ is distinct from $k$ and $l$. 

Therefore, the triple sum actually reduces to two double sums corresponding to $i=k$ and $i=l$. Explicitly, the expression in the curly bracket takes the form
\begin{equation}\label{eq5.H}
\begin{aligned}
& \sum_{1\le k<l\le N}
\left(t^{N-k}(q-1) + 
t^{k-N}(1-q)\right)\x^{\de -\eps_k+\eps_l}\\
+ & \sum_{1\le k<l\le N}
\left(t^{N-l}(q-1) + 
t^{l-N}(1-q)\right)\x^{\de+ \eps_k-\eps_l},
\end{aligned}
\end{equation}
where the factor $1-q$ comes from $q(q^{-1}-1)$. 

The resulting expression is a linear combination of Laurent monomials. However, after antisymmetrization all monomials containing coinciding superscripts will be killed. 

Examine the first sum in \eqref{eq5.H}. Here coinciding superscrips arise whenever $l-k\ge2$, so that we keep only the pairs of indices with $l=k+1$. Their contribution is
\begin{equation}\label{eq5.I}
\left(\sum_ {k=1}^{N-1}
\left(t^{N-k}(q-1)+ 
t^{k-N}(1-q)\right)\right)\x^{\de -\eps_k+\eps_{k+1}}
\end{equation}

Examine now the second sum in \eqref{eq5.H}. Here the only monomial $\x^{\de+\eps_k-\eps_l}$ without coinciding superscripts is the one with $k=1$ and $l=N$. But for $l=N$, the corresponding coefficient vanishes, because
$$
(t^{N-l}(q-1) + t^{l-N}(1-q))\big|_{l=N}=(q-1)+(1-q)=0.
$$ 
Thus, the second sum gives no contribution at all, so that we are left with \eqref{eq5.I}.

Next, we observe that 
$$
\Anti_N \x^{\de -\eps_k+\eps_{k+1}}=-\Anti_N\x^\de=-V_N.
$$
for each $k=1,\dots,N-1$. Indeed, this holds because $\x^{\de -\eps_k+\eps_{k+1}}$ is obtained from $\x^\de$ be switching the variables $x_k$ and $x_{k+1}$. 

It follows that
$$
\CT(D_N e_{2,N})=- \, \frac{q t^{N-1}}{a|b|} \left(\sum_ {k=1}^{N-1}
\left(t^{N-k}(q-1)+ t^{k-N}(1-q)\right)\right).
$$
Computing the sum we obtain \eqref{eq5.G2}. 
\end{proof}

The next lemma is a corollary of Lemmas \ref{lemma5.B} and \ref{lemma5.C}.

\begin{lemma}\label{lemma5.D}
Assume $f\in\Sym(N)$ has a local minimum on $\wt\Om_N(q,t;a,b)$ at the point $0^N$. Then $D_N f(0^N)\ge0$.
\end{lemma}

\begin{proof}
Expand $f$ on homogeneous components, $f=\sum_{k\ge0}f_k$. From Lemma \ref{lemma5.B} we know that $f_1=0$. Next, $D_N$ decreases degree at most by $2$. Therefore, $D_Nf_k(0^N)=0$ for all $k\ge3$, so that 
$$
D_Nf(0^N)=D_Nf_2(0^N)=\CT(D_N f_2).
$$
If $f_2=0$, then this quantity equals $0$. Assume $f_2\ne0$ and write $f_2=c'_2 p_{2,N}+c''_2 e_{2,N}$, as in Lemma \ref{lemma5.C}. From this lemma we know that $c'_2>0$ and $u:=c''_2/c'_2 $ satisfies the inequality $u\le(1+q)/\sqrt q$. It suffices to prove that this inequality implies
$$
\CT(D_N p_{2,N})+u \CT(D_Ne_{2,N})\ge0.
$$
The two constant terms $\CT(\cdot)$ are computed in Lemma \ref{lemma5.C}, formulas \eqref{eq5.G1} and \eqref{eq5.G2}.  The first term is positive and the second term is negative. The common factor $C_N(q,t)>0$ defined in \eqref{eq5.J} is strictly positive. Therefore, by virtue of \eqref{eq5.G1} and \eqref{eq5.G2}, the desired inequality is reduced to the following one:
$$
(1+q)(t^{1-N}q^{-1}-1)\ge \frac{1+q}{\sqrt q}(t^{1-N}-1).
$$
But this is further reduced to
$$
t^{1-N}(q^{-1/2}-1)\ge q^{1/2}-1.
$$
The latter inequality holds true because the left hand side is positive while the right-hand side is negative.
\end{proof}

Lemma \ref{lemma5.D} proves the claim of Proposition \ref{prop5.A} in the particular case when $\wt X=0^N$. In the next lemma we consider another particular case.

\begin{lemma}\label{lemma5.E}
Assume $f\in\Sym(N)$ attains its minimum on $\wt\Om_N(q,t;a,b)$ at a configuration $Y\in\Om_N(q,t;a,b)\subset \wt\Om_N(q,t;a,b)$. Then $D_N f(Y)\ge0$. 
\end{lemma}

\begin{proof}
Write $Y=\{y_1,\dots,y_N\}$. Because $Y\in\Om_N(q,t;a,b)$, all $y_i$'s are pairwise distinct and we can use formula \eqref{eq5.B} directly:
$$
D_N f(Y)=\sum_{i=1}^N\left(\frac{S_{t,i}^+ V_N(Y)}{V_N(Y)}\si_N^+(y_i)(S_{q,i}^+f(Y)-f(Y))+\frac{S_{t,i}^- V_N(Y)}{V_N(Y)}\si^-_N(y_i)(S_{q,i}^-f(Y)-f(Y))\right). 
$$
For definiteness, we may assume $y_1>\dots >y_N$. 

Let us examine the signs of all terms. 

\emph{Step}\/ 1. Let us show that for all $i=1,\dots,N$ 
$$
\frac{S_{t,i}^\pm V_N(Y)}{V_N(Y)}\ge0.
$$
Indeed, since $y_1>\dots>y_N$, we have $V_N(Y)>0$ and the sign of the numerator is that of the expression
$$
\prod_{k<i}(y_k-y_i t^{\pm1})\cdot \prod_{l>i}(y_i t^{\pm1}-y_l).
$$
It suffices to check that all factors are nonnegative, which reduces to the inequalities 
$$
y_{i-1}-y_it^{-1}\ge0, \quad \text{for $i\ge2$}; \qquad y_i t-y_{i+1}\ge0, \quad \text{for $i\le N-1$}.
$$
But these inequalities hold by the definition of the set $\Om_N(q,t;a,b)$: in the logarithmic scale, the distances between the particles are no less than $\log(1/t)$. 

\emph{Step}\/ 2. Next, for all $i=1,\dots,N$ we have 
$$
\si_N^+(y_i)>0, \qquad S_{q,i}^\pm f(Y)-f(Y)\ge0. 
$$
Indeed, the first inequality follows from the definition of $\si_N^+(\cdot)$ (see \eqref{eq5.A1} and the constraints on $(\abcd)$), and the second inequality holds due to our assumption on $f$. Consequently, 
$$
\si_N^+(y_i)(S_{q,i}^+ f(Y)-f(Y))\ge0, \quad i=1,\dots,N. 
$$

\emph{Step}\/ 3. Finally, we also have 
$$
\si_N^-(y_i)(S_{q,i}^-f(Y)-f(Y))\ge0, \quad i=1,\dots,N.
$$
Here the argument is similar, with the only exception of the case when $i=1$ and $y_1=a^{-1}q$ or when $i=N$ and  $y_N=b^{-1}q$. But then we use the fact that 
$$
\si^-_N(a^{-1}q)=\si^-_N(b^{-1}q)=0,
$$
which follows from \eqref{eq5.A2}.

We conclude that the whole expression for $D_N f(Y)$ is nonnegative. 
\end{proof}

\begin{proof}[End of proof of Proposition \ref{prop5.A}]
In the last two lemmas, we have examined two particular cases, $\wt X=0^N$ and $\wt X=Y\in\Om_N(q,t;a,b)$. Assume now that $\wt X$ is not in $\Om_N(q,t;a,b)$ and  is distinct from $0^N$. This means that $\wt X$ can be represented in the following form:
$$
\wt X=\{0^n\}\sqcup Y, \quad \text{where} \quad Y\in\Om_m(q,t;a,b), \; n>0, \; m>0,\; n+m=N. 
$$ 

Write $Y=\{y_j: 1\le j\le m\}$ and let $X^0=\{x_i: 1\le i\le n\}$ be an $n$-tuple of variables. We also consider the $N$-tuple 
$$
X=X^0\sqcup Y:=\{x_1,\dots,x_n, y_1,\dots,y_m\}.
$$
We have
\begin{equation}
D_N f(\wt X)=\lim_{X^0\to 0^n} D_N f(X),
\end{equation}
where $X^0\to0^n$ means that all $x_i$'s go to $0$. 

We split the operator $D_N$ into two parts,
$$
D_N=D_{X^0}+D_{Y},
$$
where $D_{X^0}$ and $D_Y$ assemble the terms containing the $q$-shifts on $x_i$'s and on $y_j$'s, respectively. To distinguish between the two groups of variables, below we use a more detailed notation for the shift operators. 

According to the splitting of $D_N$ we have
$$
D_N f(X)=D_{X^0}f(X)+D_Y f(X).
$$
We are going to show that both $D_{X^0}f(X)$ and $D_Y f(X)$ have limits as $X^0\to 0^n$ and these limit values are nonnegative. This will prove the proposition. 

Examine $D_Y f(X)$ first. We have
\begin{align*}
D_Y f(X)=&\sum_{j=1}^m\left(\prod_{i=1}^n \frac{y_jt-x_i}{y_j-x_i}\right)\left(\prod_{r:\, r\ne j}\frac{y_jt-y_r}{y_j-y_r}\right) \si^+_N(y_j) (S^+_{q,y_j}f(X)-f(X))\\
+&\sum_{j=1}^m\left(\prod_{i=1}^n \frac{y_jt^{-1}-x_i}{y_j-x_i}\right)\left(\prod_{r:\, r\ne j}\frac{y_jt^{-1}-y_r}{y_j-y_r}\right) \si^-_N(y_j) (S^-_{q,y_j}f(X)-f(X)).
\end{align*}
As $X^0\to 0^n$, this expression tends to 
\begin{align*}
&\sum_{j=1}^m t^n \left(\prod_{r:\, r\ne j}\frac{y_jt-y_r}{y_j-y_r}\right) \si^+_N(y_j) (S^+_{q,y_j}f(\wt X)-f(\wt X))\\
+&\sum_{j=1}^m t^{-n}\left(\prod_{r:\, r\ne j}\frac{y_jt^{-1}-y_r}{y_j-y_r}\right) \si^-_N(y_j) (S^-_{q,y_j}f(\wt X)-f\wt (X)).
\end{align*}
Using the fact that the value $f(\wt X)$ is the minimum we obtain
$$
\si^\pm_N(y_j)(S^\pm_{q,y_j}f(\wt X)-f(\wt X))\ge0, \quad j=1,\dots,m.
$$
The argument here is the same as on steps 2 and 3 from Lemma \ref{lemma5.E}). Next, the products over $r$ are nonnegative, too (the same argument as on step 1 from Lemma \ref{lemma5.E}). Thus, the whole expression is nonnegative.

We now turn to $D_{X^0}f(X)$. Because both $D_N f(X)$ and $D_Y f(X)$ have limits as $X^0\to 0^n$, the limit of $D_{X^0}f(X)$ exists, too. 

The configuration $Y$ being fixed, we set 
$$
g(x_1,\dots,x_n):=f(x_1,\dots,x_n,y_1,\dots,y_m).
$$
This is a symmetric polynomial and $f(X)=g(X^0)$. 

We can represent $D_{X^0}f(X)$ as the sum of two expressions: one is 
\begin{equation}\label{eq5.K}
\begin{aligned}
 &\sum_{i=1}^m\left(\prod_{r:\, r\ne i}\frac{x_it-x_r}{x_i-x_r}\right) \si^+_N(x_i) (S^+_{q,x_i}g(X^0)-g(X^0))\\
+&\sum_{i=1}^m\left(\prod_{r:\, r\ne i}\frac{x_it^{-1}-x_r}{x_i-x_r}\right) \si^-_N(x_i) (S^-_{q,x_i}g(X^0)-g(X^0))
\end{aligned}
\end{equation}
and the other is
\begin{equation}\label{eq5.L}
\begin{aligned}
  &\sum_{i=1}^n\left(\prod_{j=1}^m \frac{x_it-y_j}{x_i-y_j}-1\right)\left(\prod_{r:\, r\ne i}\frac{x_it-x_r}{x_i-x_r}\right) \si^+_N(x_i) (S^+_{q,x_i}g(X^0)-g(X^0))\\
+&\sum_{i=1}^n \left(\prod_{j=1}^m \frac{x_it^{-1}-y_j}{x_i-y_j}-1\right)\left(\prod_{r:\, r\ne i}\frac{x_it^{-1}-x_r}{x_i-x_r}\right) \si^-_N(x_i) (S^-_{q,x_i}g(X^0)-g(X^0)).
\end{aligned}
\end{equation}

Observe that \eqref{eq5.K} coincides with $t^m D_n g(X^0)$, where the factor $t^m=t^{N-n}$ arises because $\si^\pm_N(x)=t^{N-n}\si^\pm_n(x)$, see \eqref{eq5.A1} and \eqref{eq5.A2}. Therefore, as $X^0$ approaches $0^n$, the expression \eqref{eq5.K} tends to $t^n D_n g(0^n)$. This implies in turn that the expression \eqref{eq5.L} has a limit, too. 

We are going to show that $D_n g(0^n)\ge0$ and that the limit of \eqref{eq5.L} is $0$. This will prove the proposition.

Among the $y_j$'s, there is some number of positive points and some number of negative points. Denote these two numbers by  $m^+$ and $m^-$, respectively. Next, introduce two parameters $\wt a$ and $\wt b$ such that
$$
\wt a^{-1}:= t^{m^+}a^{-1}, \qquad \wt b^{-1}:=t^{m^-} b^{-1}.
$$
If $X^0\in \Om_n(q,t; \wt a, \wt b)$ and all $x_i$'s are close enough to $0$, then $X\in \Om_N(q,t;a,b)$. 

It follows that the polynomial $g(x_1,\dots,x_n)$, viewed as a function on $\wt\Om_n(q,t; \wt a, \wt b)$, has a local minimum at $0^n$. Thus, $g$ satisfies the assumption of Lemma \ref{lemma5.D} (with $N$, $a$, and $b$ replaced by $n$, $\wt a$, and $\wt b$, respectively). By virtue of this lemma, $D_n g(0^n)\ge0$.

It remains to show that the limit of \eqref{eq5.L} is $0$.  Expand $g$ into the sum of its homogeneous components,  $g=g_0+g_1+g_2+\dots$\,. The value of the constant term $g_0$ is irrelevant, because it does not affect the quantities $S^\pm_{q,x_i}g(X^0)-g(X^0)$. Thus, we may assume $g_0=0$. Next, because $g$ satisfies the assumption of Lemma \ref{lemma5.B} (again, with a suitable change of the parameters), we have $g_1=0$. 

The limit value of \eqref{eq5.L} does not depend on the way in which the $n$ variables $x_1,\dots,x_n$ approach $0$, so we have the freedom to specify it as we wish. Using this freedom we will assume that
$$
(x_1,\dots,x_n)=(\eps x'_1,\dots,\eps x'_n), 
$$
where $(x'_1,\dots,x'_n)$ is fixed, the $x'_i$ are pairwise distinct, and $\eps=q^m$ goes to $0$. 

Let us substitute these $x_i$'s into \eqref{eq5.L} and evaluate the order of all terms with respect to $\eps\to0$. 

$\bullet$ Since $g_1=0$, we have $S^\pm_{q,x_i}g(X^0)-g(X^0)=O(\eps^2)$.

$\bullet$ Since $\si^\pm_N(x_i)$ is a linear combination of $1, x_i^{-1}, x_i^{-2}$,  we have $\si^\pm_N(x_i)=O(\eps^{-2})$. 

$\bullet$ From this we see that all terms of the form $\si_N^\pm(x_i)(S^\pm_{q,x_i}g(X^0)-g(X^0))$ are $O(1)$. 

$\bullet$ Next, the products over $r$ do not depend on $\eps$.

$\bullet$ Finally, for any $i$, 
$$
\prod_{j=1}^m \frac{x_it^{\pm1}-y_j}{x_i-y_j}-1 =\prod_{j=1}^m \frac{\eps x'_i t^{\pm1}-y_j}{\eps x'_i-y_j}-1=O(\eps).
$$ 

It follows that the whole expression is of order $O(\eps)$ and hence tends to $0$ as $\eps\to0$. This completes the proof. 
\end{proof}

\subsection{Markov dynamics on $\wt\Om_N(q,t;a,b)$}

Proposition \ref{prop5.A} leads to a generalization of Theorem \ref{thm3.A}. As above, $(q,t;\abcd)$ is a fixed $6$-tuple of parameters such that $q,t\in(0,1)$, $b<0<a$, and $c=\bar d\in\C\setminus\R$. Recall that the operator $D_N$, depending on these parameters, is initially defined on the space $\Sym(N)$ by \eqref{eq5.B} and then it is transferred to $C(\wt\Om_N(q,t;a,b))$ by embedding $\Sym(N)$ into $C(\wt\Om_N(q,t;a,b))$ as a dense subspace.

\begin{theorem}\label{thm5.A} 
The operator $D_N$ serves as a pre-generator of a Feller semigroup on the space $C(\wt\Om_N(q,t;a,b))$.
\end{theorem}

Combining this result with the general theorem from section \ref{sect2.5} we obtain that $D_N$ gives rise to a Markov process on the space $\wt\Om_N(q,t;a,b)$.

\begin{proof}
We apply the criterion stated in section \ref{sect2.4}. 

Condition (iii) holds by virtue of Proposition \ref{prop5.A}. 

Condition (i) is obvious. Indeed, $D_N$ is defined on the dense subspace $\Sym(N)\subset C(\wt\Om_N(q,t,a,b))$.

Finally, condition (ii) is checked as follows. Recall that $\Y$ is our notation for the set of partitions, which we identify with the corresponding Young diagrams. Let $\Y(N)\subset\Y$ denote the subset of partitions with the length at most $N$. Consider the basis $\{m_{\la\mid N}:\la\in\Y(N)\}$ in $\Sym(N)$ formed by the monomial symmetric polynomials. Let us equip $\Y(N)$ with the lexicographic order. 

From \cite[Proposition 4.2 and formula (3.7)]{S} it follows that $D_N$ has triangular form in the basis $\{m_{\la\mid N}\}$, with nonpositive entries on the diagonal. 

Indeed, for each $\la\in\Y(N)$
$$
D_N\, m_{\la\mid N} = \m_{\la\mid N}(q,t;a,b;c,d)\, m_{\la\mid N}+ \text{lower terms in lexicographic order}, 
$$
where 
\begin{equation}\label{eq5.M}
\m_{\la\mid N}(\la;q,t;\abcd):=-\, \sum_{i=1}^N\left(\frac{cdq}{ab} t^{2N-i-1}(q^{\la_i}-1)+t^{\la_i-1}(q^{-\la_i}-1)\right).
\end{equation}
(To deduce this result from \cite[Proposition 4.2]{S} one should take into account a difference in notation, see section \ref{sect5.1}. Actually, that proposition gives a stronger result, but for our purpose the above formulation suffices.) 
Next, because $cd>0$ and $ab<0$, the quantity $\m_{\la\mid N}(q,t;\abcd)$ is strictly negative for all $\la\ne\varnothing$ (where $\varnothing$ denotes the empty Young diagram), and vanishes for $\la=\varnothing$: indeed, both summands inside the big round brackets are strictly positive for $\la\ne\varnothing$, and vanish for $\la=\varnothing$.  

We conclude that for any $r>0$, the operator $r-D_N$ on $\Sym(N)$ is invertible. This implies the desired condition (ii). 
\end{proof}

\section{$N$-variate big $q$-Jacobi polynomials and their analogs in the algebra of symmetric functions}\label{sect6}

Here we collect a few necessary results from \cite{S}, \cite{SK}, and \cite{Ols-2021b}. 

\subsection{The polynomials $\varphi_{\la\mid N}(x_1,\dots,x_N; q,t;\abcd)$}

As before, we assume that the parameters $q$ and $t$ are in $(0,1)$. For $\la\in\Y(N)$, we denote by $P_{\la\mid N}(-;q,t)$ the $N$-variate Macdonald polynomial with the parameters $(q,t)$ and the index $\la\in\Y(N)$ (\cite[Ch. VI]{Mac-1995}).  The Macdonald polynomials form a homogeneous basis in $\Sym(N)$. The degree of $P_{\la\mid N}(-;q,t)$ equals $|\la|$, the number of boxes in the diagram $\la$.

Recall our  assumptions on the parameters $\abcd$:  $b<0<a$ and $c=\ov d\in\C\setminus\R$.  Recall that the $q$-difference operator $D_N$ acting on $\Sym(N)$ is defined in \eqref{eq5.B}.

Recall also that the quantities $\m_{\la\mid N}(q,t;\abcd)$ are defined in \eqref{eq5.M}. 

Given a Young diagram $\la\in\Y$ we introduce the \emph{generalized Pochhammer symbol} 
\begin{equation}\label{eq6.B}
(z;q,t)_\la:=\prod_{(i,j)\in\la} (1-zt^{1-i}q^{j-1}), \quad z\in\C, 
\end{equation}
where $(i,j)\in\la$ means that $\la$ contains the box on the intersection of row $i$ and column $j$. 

Finally, the notation $\nu\subseteq\la$ for Young diagrams $\la$ means that $\nu$ is contained in $\la$. 

\begin{proposition}\label{prop6.A}
Let $(q,t;\abcd)$ be a $6$-tuple of parameters satisfying the constraints listed above. For each $N$, in the algebra $\Sym(N)$, there exists an inhomogeneous basis of polynomials 
\begin{equation}\label{eq6.A}
\varphi_{\la\mid N}(-; q,t;\abcd)=\varphi_{\la\mid N}(x_1,\dots,x_N; q,t;\abcd), \qquad \la\in\Y(N),
\end{equation}
with the following properties.

$\bullet$ The polynomials \eqref{eq6.A} are eigenfunctions of the operator $D_N$: one has 
\begin{equation}\label{eq6.F}
D_N\varphi_{\la\mid N}(-; q,t;\abcd)=\m_{\la\mid N}(q,t;\abcd) \varphi_{\la\mid N}(-; q,t;\abcd),
\end{equation}
where the eigenvalues $\m_{\la\mid N}(q,t;\abcd)$ are given by \eqref{eq5.M}.

$\bullet$ The expansion of $\varphi_{\la\mid N}(-; q,t;\abcd)$ in the basis of Macdonald polynomials can be written in the form
\begin{equation}\label{eq6.C}
\varphi_{\la\mid N}(-;q,t;\abcd)=\sum_{\nu: \, \nu\subseteq\la}\frac{(t^N;q,t)_\la}{(t^N;q,t)_\nu}\pi_N(\la,\nu;q,t;\abcd) P_{\nu\mid N}(-;q,t),
\end{equation}
where $\pi_N(\la,\nu;q,t;\abcd)$ are certain coefficients, such that
\begin{equation}\label{eq6.D}
\pi_N(\la,\la;q,t;\abcd)=1
\end{equation}
and the following relations hold
\begin{equation}\label{eq6.E}
\pi_N(\la,\nu;q,t;\abcd)=\pi_{N+1}(\la,\nu;q,t; a,b, ct^{-1}, dt^{-1}),  \qquad \la,\nu\in\Y(N).
\end{equation}
\end{proposition}

The polynomials $\varphi_{\la\mid N}(-; q,t;\abcd)$ are called the $N$-variate \emph{big $q$-Jacobi polynomials}.\footnote{In the limit as $q$ and $t$ go to $1$ keeping the quantity $\log_q t$ fixed, these polynomials degenerate into Heckman--Opdam's Jacobi polynomials \cite{HO}.} They were introduced and studied in the works of Stokmant \cite{S} and Stokman--Koornwinder \cite{SK}. Recall that the notation of the parameters in these papers differs from our notation (see \eqref{eq5.C}). When $N=1$, the dependence on $t$ disappears and one gets the classical big $q$-Jacobi orthogonal polynomials studied by Andrews and Askey \cite{AA} (see also Koekoek-Swarttouw\cite[sect. 3.5]{KS}, Koornwinder \cite[sect. 14.5]{Koo-2014}). 

In the special case $t=q$, the $N$-variate polynomials can be expressed through the univariate polynomials in a simple way. But this method does not work for $t\ne q$.  

Let us explain how to connect Proposition \ref{prop6.A} with results from \cite{S}, \cite{SK}, and \cite{Ols-2021b}. 

In \cite[Definition 5.1]{S}, the polynomials $\varphi_{\la\mid N}(-; q,t;\abcd)$ are specified by the following characteristic properties.

$\bullet$ \emph{Triangularity}: $\varphi_{\la\mid N}(-; q,t;\abcd)$ differs from $m_{\la\mid N}$ by a linear combination of polynomials  $m_{\nu\mid N}$ with $\mu<\la$, where ``$<$'' stands for a partial order consistent with the lexicographic order. 

$\bullet$ \emph{Orthogonality}: $\varphi_{\la\mid N}(-; q,t;\abcd)$ is orthogonal to all polynomials of the form $m_{\nu\mid N}$,  $\nu<\la$, with respect to a certain scalar product  in $\Sym(N)$, which depends on the parameters  (this scalar product is given by a measure which we denote by $M_N(-;q,t;\abcd)$, see Proposition \ref{prop6.C} below). 

The fact that the polynomials introduced in this way are eigenfunctions of a partial $q$-difference operator (in our notation, $-D_N$) is established in \cite[Theorem 5.7, part (1)]{S}.  This gives the result stated in \eqref{eq6.F}. 

The paper \cite{SK} is devoted to limit transitions. It also contains a summary of the results of \cite{S} and some valuable additional remarks, see e.g. \cite[section 2]{SK}.  For our purposes, especially important is \cite[Theorem 5.1, claim (1)]{SK}. It explains how to obtain the big $q$-Jacobi polynomials from the Koornwinder polynomials through a limiting procedure:
\begin{equation}\label{eq6.T}
\varphi_{\la\mid N}(x_1,\dots,x_N;q,t;\abcd)=\lim_{\eps\to0}\eps^{|\la|} K^{(N)}_\la\left(\dfrac{x_1}{\eps},\dots,\dfrac{x_N}{\eps}; q,t; c\eps, d\eps; \dfrac q{a\eps}, \dfrac q{b\eps}\right),
\end{equation}
where $K^{(N)}_\la(-;q,t;t_0,t_1,t_2,t_3)$ are the Koornwinder polynomials (here we use the notation of \cite{R}). 

In \cite[Proposition 3.1]{Ols-2021b}, the existence of the limit on the right-hand side is obtained in a different way, by making use of Okounkov's BC-type interpolation polynomials \cite{Ok} and Rains' work \cite{R}; and then \eqref{eq6.T} is taken as the initial definition. This leads to the expansion \eqref{eq6.C}, with an explicit (albeit complicated) expression for the coefficients. See \cite[Propositions 3.1 and 3.2, and Corollary 3.3]{Ols-2021b}. 

The approach of \cite{Ols-2021b} also provides an alternative proof of the formulas \eqref{eq6.F} and \eqref{eq5.M}: once the degeneration of Koornwinder polynomials to big $q$-Jacobi polynomials is established, these formulas can be derived from the analogous formulas for the Koornwinder polynomials. 

Finally, let us comment on the last two formulas of Proposition \ref{prop6.A}. The equality \eqref{eq6.D} shows that the top degree homogeneous component of $\varphi_{\la\mid N}(-; q,t;\abcd)$ is the Macdonald polynomial $P_\la(-;q,t)$; of course, this fact can also be derived from the results of \cite{S}, \cite{SK}. Next, the relation \eqref{eq6.E} is of fundamental importance to us as it leads to  the definition of the big $q$-Jacobi symmetric functions, see the next subsection.

\subsection{The symmetric functions $\Phi_\la(-;q,t;\alde)$}\label{sect6.2}

From now on we make $c$ and $d$ depending on $N$ and change accordingly our notation. Namely, we will assume that
\begin{equation}\label{eq6.L}
(\abcd)= (\al,\be;\ga t^{1-N}, \de t^{1-N}), \qquad N\in\Z_{\ge1},
\end{equation}
where $(\alde)$ is a fixed $4$-tuple satisfying the same constraints as before:
$$
\be<0<\al, \qquad \ga=\ov\de\in\C\setminus\R.
$$ 

\begin{definition}\label{def6.A}
Let $\la\supseteq\nu$ be a pair of Young diagrams. Observe that, by virtue of \eqref{eq6.E}, the quantity $\pi_N(\la,\nu;q,t;\al,\be,\ga t^{1-N}, \de t^{1-N})$ is \emph{stable} in the sense that it does not depend on $N$. Let us denote this stable quantity by $\pi(\la,\nu;q,t;\alde)$. 
\end{definition}

In this notation, \eqref{eq6.C} takes the form
\begin{equation}\label{eq6.G}
\varphi_{\la\mid N}(-;q,t;\al,\be,\ga t^{1-N}, \de t^{1-N})=\sum_{\nu: \, \nu\subseteq\la}\frac{(t^N;q,t)_\la}{(t^N;q,t)_\nu}\pi(\la,\nu;q,t;\alde) P_{\nu\mid N}(-;q,t)
\end{equation}
or else
\begin{equation}\label{eq6.H}
\frac{\varphi_{\la\mid N}(-;q,t;\al,\be,\ga t^{1-N}, \de t^{1-N})}{(t^N;q,t)_\la}=\sum_{\nu: \, \nu\subseteq\la}\pi(\la,\nu;q,t;\alde) \frac{P_{\nu\mid N}(-;q,t)}{(t^N;q,t)_\nu}.
\end{equation}

We refer to the polynomials on the left-hand and right-hand sides of \eqref{eq6.H} as to the \emph{renormalized} polynomials (big $q$-Jacobi and Macdonald, respectively). 

Let us denote by $P_\nu(-;q,t)$ the Macdonald symmetric \emph{functions} indexed by arbitrary Young diagrams $\nu\in\Y$ (see \cite[ch. VI]{Mac-1995}).

\begin{definition}\label{def6.C}
The \emph{big $q$-Jacobi symmetric functions} are the following elements of $\Sym$ indexed by arbitrary Young diagrams $\la\in\Y$: 
\begin{equation}\label{eq6.I}
\Phi^{q,t;\alde}_\la(-):= \sum_{\nu: \, \nu\subseteq\la}\pi(\la,\nu;q,t;\alde) P_\nu(-;q,t).
\end{equation}
\end{definition}

This definition is taken from \cite[section 7.1]{Ols-2021b}. In the special case $q=t$ it was given in \cite{Ols-2017}. Because the coefficient $\pi(\la,\nu;q,t;\alde)$ equals $1$ for $\nu=\la$, the top degree homogeneous component of $\Phi^{q,t;\alde}_\la(-)$ equals $P_\la(-;q,t)$. It follows that the big $q$-Jacobi symmetric functions form an inhomogeneous basis in $\Sym$, consistent with the filtration by degree. 

Denote by $\Sym_{\le d}$ the subspace in $\Sym$ formed by the elements of degree at most $d$, where $d=0,1,2,\dots$\,. In the similar way we define the subspaces $\Sym_{\le d}(N)\subset \Sym(N)$. These subspaces have finite dimension and the canonical projection  $\Sym\to\Sym(N)$ determines a projection $\Sym_{\le d}\to\Sym_{\le d}(N)$. For $N\ge d$ the latter projection is a linear isomorphism and hence has the inverse:
$$
\iota_{d,N}: \Sym_{\le d}(N)\to \Sym_{\le d}, \qquad N\ge d.
$$
Evidently, $\iota_{d+1,N}$ extends $\iota_{d,N}$ for every $N\ge d+1$. 

The next definition is taken from \cite[Definition 2.4]{Ols-2017}.

\begin{definition}\label{def6.B}
Let us say that a sequence $\{F_N\in\Sym(N): N\ge N_0\}$ \emph{converges} to an element $F\in\Sym$ if $\sup_N\deg F_N<\infty$ and for $d$ large enough
$$
\lim_{N\to\infty}\iota_{d,N}(F_N)\to F
$$
in the finite-dimensional space $\Sym_{\le d}$.  Then we write $F_N\to F$ or $\lim_N F_N=F$.
\end{definition}

The next proposition is an extension of \cite[Theorem 3.4 (i)]{Ols-2017}.

\begin{proposition}\label{prop6.B}
For each $\la\in\Y$
\begin{equation}\label{eq6.J}
\varphi_{\la\mid N}(-;q,t;\al,\be,\ga t^{1-N}, \de t^{1-N}) \to \Phi^{q,t;\alde}_\la(-)
\end{equation}
in the sense of Definition \ref{def6.B}.
\end{proposition}

\begin{proof}
Due to the stability property of the Macdonald polynomials, given $\nu\in\Y$ and $d\ge|\nu|$, for all $N$ large enough
$$
\iota_{d,N}(P_{\nu\mid N}(-;q,t))=P_\nu(-;q,t).
$$
It follows that 
\begin{equation}\label{eq6.K}
P_{\nu\mid N}(-;q,t)\to P_\nu(-;q,t), \qquad \forall\nu\in\Y.
\end{equation}

On the other hand, $(t^N;q,t)_\mu\to 1$ for any fixed diagram $\mu\in\Y$ (because $t\in(0,1)$ and hence $t^N\to0$). Therefore, the coefficients in the expansion \eqref{eq6.G} converge to the corresponding coefficients in the expansion \eqref{eq6.I}. Together with \eqref{eq6.K} this implies \eqref{eq6.J} (cf. \cite[Proposition 2.5]{Ols-2017}).
\end{proof}

\subsection{Orthogonality measures}

Return for a moment to our previous parameters $(\abcd)$. 

\begin{proposition}\label{prop6.C}
On the set\/ $\wt\Om_N(q,t;\abcd)$, there exists a unique probability measure $M_N(-; q,t;\abcd)$ such that the big $q$-Jacobi polynomials $\varphi_{\la\mid N}(-; q,t;\abcd)$ are orthogonal with respect to it. That is, for any partitions $\la,\, \mu\in\Y(N)$
\begin{gather*}
\sum_{X\in\wt\Om_N(q,t;\abcd)} \varphi_{\la\mid N}(X; q,t;\abcd)\varphi_{\mu\mid N}(X; q,t;\abcd)M_N(X; q,t;\abcd)\\
=\begin{cases} \text{\rm a positive constant $h_{\la\mid N}(q,t;\abcd)$}, & \mu=\la, \\
0, & \mu\ne\la.
\end{cases}
\end{gather*} 
\end{proposition}

This is a reformulation of Stokman's result  \cite[Theorem 5.7 (1)]{S}, see \cite[Theorem 5.2]{Ols-2021b}. 

We call $M_N(-; q,t;\abcd)$ the \emph{orthogonality measure} for the $N$-variate big $q$-Jacobi polynomials. It is in fact supported by the subset $\Om_N(q,t;a,b)\subset\wt\Om_N(q,t;a,b)$. Adopting the notation of \cite[Definition 6.1]{Ols-2021b}, we denote by $\MM_N^{q,t;\alde}$ the orthogonality measure with the shifted parameters \eqref{eq6.L}.  

Recall that $\wt\Om_N(q,t;\al,\be)$ is contained in the space $\wt\Om(q,t;\al,\be)$, for any $N$. 

\begin{proposition}\label{prop6.D}
As $N\to\infty$, the measures\/ $\MM_N^{q,t;\alde}$ weakly converge to a probability measure $\MM_\infty^{q,t;\alde}$ on $\wt\Om(q,t;\al,\be)$, which is in fact concentrated on the subset $\Om_\infty(q,t;\al,\be)\subset\wt\Om(q,t;\al,\be)$.
\end{proposition}

This result is contained in \cite[Theorem 6.5]{Ols-2021b}.

Recall that Lemma \ref{lemma4.A} allows us to treat $\Sym$ as a dense subspace of $C(\wt\Om(q,t;\al,\be))$. Thus, we may consider symmetric functions as continuous functions on $\wt\Om(q,t;\al,\be)$. 

\begin{remark}\label{rem6.A}
Obviously, the limit measure $\MM_\infty^{q,t;\alde}$ satisfies the relations
\begin{equation}\label{eq6.S}
\langle 1,\; \MM_\infty^{q,t;\alde}\rangle=1, \qquad \langle \Phi^{q,t;\alde}_\la(-),\; \MM_\infty^{q,t;\alde}\rangle=0\quad \text{for $\la\ne\varnothing$},
\end{equation}
where the angular brackets denote the canonical pairing between functions and measures. Morover, these relations characterize $\MM_\infty^{q,t;\alde}$ uniquely. 
\end{remark}

We adopt the following notation from \cite[section 2]{R}:
\begin{equation}\label{eq6.U}
C^+_\la(x;q,t):=\prod_{(i,j)\in\la} \left(1-q^{\la_i+j-1}t^{2-\la'_j-i}x\right), \quad 
C^-_\la(x;q,t):=\prod_{(i,j)\in\la} \left(1-q^{\la_i-j}t^{\la'_j-i}x\right),
\end{equation}
where $\la\in\Y$ and $\la'$ is the transposed Young diagram. We also set
$$
\mathfrak s:=\frac{\ga\de q}{\al\be}, \quad n(\la):=\sum_i  (i-1)\la_i, \quad 2\la\cup 2\la:=(2\la_1,2\la_1, 2\la_2,2\la_2, 2\la_3, 2\la_3,\dots).
$$

\begin{proposition}\label{prop6.E}
The limit measure $\MM_\infty^{q,t;\alde}$ serves as the orthogonality measure for the big $q$-Jacobi symmetric functions $\Phi^{q,t;\alde}_\la(-)$ in the sense that for any partitions $\la,\mu\in\Y$
\begin{equation}\label{eq6.O}
\langle\Phi_\la^{q,t;\alde}(-)\Phi_\mu^{q,t;\alde}(-), \; \MM_\infty^{q,t;\alde}\rangle
=\begin{cases} h_\la(q,t;\alde), & \mu=\la, \\
0, & \mu\ne\la,
\end{cases}
\end{equation}
where 
\begin{gather}
h_\la(q,t;\alde):=\dfrac{C^-_\la(q;q,t)}{C^-_\la(t;q,t)}\, \frac{C^+_\la(\s qt^{-1};q,t)}{C^+_\la(\s;q,t)}\, \left(\dfrac{\s q^3}{\al\be t}\right)^{|\la|}\,\dfrac{q^{2 n(\la')}}{t^{2n(\la)}}\, \frac1{(\s q;q,t)_{2\la \cup 2\la}}  \label{eq6.N1} \\
\times \left(\dfrac{\ga q}\al; q,t\right)_\la \, \left(\dfrac{\ga q}\be; q,t\right)_\la\,  \left(\dfrac{\de q}\al; q,t\right)_\la \, \left(\dfrac{\de q}\be; q,t\right)_\la. \label{eq6.N2}
\end{gather}
\end{proposition}

In the special case $t=q$, the expression for $h_\la(q,t;\alde)$ simplifies and agrees with \cite[(4.1)]{Ols-2017}. 

\begin{proof}
Let $\E_N^{q,t;\abcd}$ stand for the linear functional on $\Sym(N)$ such that 
$$
\E^{q,t;\abcd}_N(1)=1, \qquad \E^{q,t;\abcd}_N(\varphi_\la(-;q,t;\abcd))=0, \quad \la\ne\varnothing.
$$
Then the orthogonality relations from Proposition \ref{prop6.C} can can be written in the form
\begin{equation}\label{eq6.P}
\E^{q,t;\abcd}_N(\varphi_{\la\mid N}(-; q,t;\abcd)\varphi_{\mu\mid N}(-; q,t;\abcd))
=\begin{cases} h_{\la\mid N}(q,t;\abcd), & \mu=\la, \\
0, & \mu\ne\la.
\end{cases}
\end{equation} 

Likewise, let $\E^{q,t;\alde}$ stand for the linear functional on $\Sym$ such that 
$$
\E^{q,t;\alde}(1)=1, \qquad \E^{q,t;\alde}(\Phi_\la^{q,t;\alde}(-))=0, \quad \la\ne\varnothing.
$$

Next, denote by $\wt\E_N^{q,t;\alde}$ the functional $\E_N^{q,t;\abcd}$ with the shifted parameters $(\abcd)=(\al,\be;\ga t^{1-N}, \de t^{1-N})$. 

Observe that if $F_N\to F$ in the sense of Definition  \ref{def6.B}, then $\wt\E_N^{q,t;\alde}(F_N)\to \E^{q,t;\alde}(F)$ (indeed, this is a consequence of Proposition \ref{prop6.B}). It follows that
\begin{equation}\label{eq6.Q}
\E^{q,t;\alde}(\Phi_\la^{q,t;\alde}(-)\Phi_\mu^{q,t;\alde}(-))
=\begin{cases} h_\la(q,t;\alde), & \mu=\la, \\
0, & \mu\ne\la,
\end{cases}
\end{equation}
where
\begin{equation}\label{eq6.R}
h_\la(q,t;\alde):=\lim_{N\to\infty} h_{\la\mid N}(q,t;\al,\be; \ga t^{1-N}, \de t^{1-N}), \quad \la\in\Y.
\end{equation}

Let us call \eqref{eq6.P} and \eqref{eq6.Q} the \emph{formal orthogonality relations}. (These relations actually make sense for a much larger range of parameters.)

By virtue of \eqref{eq6.S}, the functional $\E^{q,t;\alde}$ coincides with integration against the measure $\MM_\infty^{q,t;\alde}$. Consequently, \eqref{eq6.Q} implies \eqref{eq6.O}. 

Now we have to explicitly compute the limit in \eqref{eq6.R}. An expression for the squared norm $h_{\la\mid N}(q,t;\al,\be; \ga t^{1-N}, \de t^{1-N})$ is contained in Stokman \cite[Theorem 7.5]{S-2000}, but it is written in the form which is not convenient for the large-$N$ transition. An easier way is to start with the formal orthogonality relations for the Koornwinder polynomials as presented in Rains \cite[Corollary 5.11 and Theorem 5.9]{R} with the use of his notation \eqref{eq6.U}; after that take the limit as $\eps\to0$ in accordance with \eqref{eq6.T}, and then  let  $N\to\infty$ according to \eqref{eq6.R}. The computation is a bit long,  but it is actually simpler than it might seem at first glance.
\end{proof}

\begin{remark}
1. The idea to deal with `formal orthogonality' is by no means new. See e.g. the concept of `virtual Koornwinder integral' in  Rains \cite[section 5]{R}.

2. The definition of the big $q$-Jacobi symmetric functions $\Phi_\la^{q,t;\alde}(-)$ in terms of their expansion on Macdonald symmetric functions (see \cite[(7.1)]{Ols-2021b}) makes sense under mild conditions on the parameters $(\alde)$: it suffices to assume that $\be<0<\al$ and $\ga\de>0$, which entails, in particular $\s<0$. 

3. Then one may ask under which additional constraints are the quantities $h_\la(q,t;\alde)$ strictly positive (or nonnegative) for all $\la\in\Y$. Note that under the above conditions, all factors in \eqref{eq6.N1} are strictly positive, so the question is reduced to the analysis of the expression \eqref{eq6.N2}. It is interesting that the form of the answer depends on whether the number $\tau:=\log_q t$ is rational or not. 
\end{remark}

\section{Limit Markov dynamics}\label{sect7}

\subsection{Generalities about stochastic links}

Suppose that $E$ and $E'$ are two measurable (= Borel) spaces, $x$ ranges over $E$, and $A\subset E'$ is an arbitrary measurable subset. Recall that a \emph{Markov kernel} relating $E$ to $E'$  is a function of two arguments $\La=\La(x,A)$ such that $\La(x,\ccdot)$ is a probability measure on $E'$ for each $x$, while $\La(\ccdot,A)$ is a measurable function on $E$ for each $A$. The kernel $\La$ induces a natural map $\Prob(E)\to \Prob(E')$, where $\Prob(\cdot)$ is our notation for the set of probability measures on a given measurable space; we denote this map as $M\mapsto M\La$. Dually, $\La$ also induces a linear map $\B(E')\to \B(E)$ of the spaces of bounded measurable functions. We denote this map by $f\mapsto \La f$; it has norm $1$ (here we assume that $\B(\cdot)$ is equipped with the supremum norm).  For more detail, see, e.g. Meyer \cite[ch. IX]{Meyer}.

Adopting the terminology of \cite{BO-2012} we call $\La$ a \emph{stochastic link} between $E$ and $E'$ and denote it as $\La:E\dasharrow E'$. If $\La': E'\dasharrow E''$ is another stochastic link, then their composition  is a stochastic link $\La\circ\La': E\dasharrow E''$. 

If $E'$ is a countable set, then $\La: E\dasharrow E'$ can be viewed simply as a function $\La(x,y)$ on $E\times E'$; if $E$ is also countable, then $\La$ is a stochastic matrix of the format $E\times E'$. 

Assume additionally that both $E$ and $E'$ are compact topological spaces with their canonical Borel structures. Then we say that $\La:E\dasharrow E'$ is \emph{Feller} if the corresponding linear map $\B( E')\to \B(E)$ sends $C(E)$ to $C(E)$; recall that $C(\cdot)$ is our notation for the space of continuous functions. 

\subsection{Stochastic links between the spaces $\wt\Om_N(q,t;\al,\be)$}

As above, we fix a quadruple of parameters $(q,t;\al,\be)$ such that $q,t\in(0,1)$ and $\be<0<\al$. Below we state a few results from \cite{Ols-2021a}. In that paper, they are in fact established for some larger sets of configurations $X$, without imposing the restriction $X\subset[\be^{-1}q, \al^{-1}q]$, which enters our definition of $\wt\Om_N(q,t;\al,\be)$ and $\wt\Om(q,t;\al,\be)$. However, taking into account this restriction is easy.  

\begin{proposition}\label{prop7.A} 
For each $N\in\Z_{\ge1}$, there exists a unique Feller stochastic link (depending on all four parameters),
$$
\wt\La^{N+1}_N: \wt\Om_{N+1}(q,t;\al,\be)\dasharrow \wt\Om_N(q,t;\al,\be),
$$
whose action on the Macdonald polynomials is given by 
\begin{equation}\label{eq7.G}
\wt\La^{N+1}_N\, \dfrac{P_{\nu\mid N}(-;q,t)}{(t^N;q,t)_\nu} =\dfrac{P_{\nu\mid N+1}(-;q,t)}{(t^{N+1};q,t)_\nu} , \qquad \forall \nu\in\Y(N).
\end{equation}
\end{proposition}

The construction of the links $\wt\La^{N+1}_N$ goes in two steps. First, one constructs a link $\La^{N+1}_N: \Om_{N+1}(q,t;\al,\be)\dasharrow \Om_N(q,t;\al,\be)$ (see \cite[section 5]{Ols-2021a}) and then one proves that it can be extended to the compactifications $\wt\Om_{N+1}(q,t;\al,\be)$ and $\wt\Om_N(q,t;\al,\be)$ (see \cite[section 6]{Ols-2021a}).

The two propositions below are contained in \cite[Theorem 8.1]{Ols-2021a}.

\begin{proposition}\label{prop7.B} 
There exist stochastic links 
$$
\wt\La^\infty_N: \wt\Om(q,t;\al,\be)\dasharrow \wt\Om_N(q,t;\al,\be), \quad N\in\Z_{\ge1},
$$
such that 
$$
\wt\La^\infty_{N+1}\circ\wt\La^{N+1}_N=\wt\La^\infty_N, \quad N\in\Z_{\ge1}.
$$
They are uniquely characterized by the property that
\begin{equation}\label{eq7.I}
\wt\La^\infty_N\, \dfrac{P_{\nu\mid N}(-;q,t)}{(t^N;q,t)_\nu} =P_\nu(-;q,t) , \qquad \forall \nu\in\Y(N).
\end{equation}
\end{proposition}

The next proposition uses the notion of \emph{boundary} for an infinite chain of spaces connected by stochastic links, see \cite[Definition 7.1]{Ols-2021a} and the explanation below.

\begin{proposition}\label{prop7.C}
The space $\wt\Om(q,t;\al,\be)$ is the boundary of the chain  $\{\wt\Om_N(q,t;\al,\be), \wt\La^{N+1}_N\}$.
\end{proposition}

This means that there is a bijection 
$$
\Prob(\wt\Om(q,t;\al,\be))\leftrightarrow \varprojlim \Prob(\wt\Om_N(q,t;\al,\be)), \qquad M\leftrightarrow \{M_N: N\in\Z_{\ge1}\},
$$
between probability measures $M\in\Prob(\wt\Om_N(q,t;\al,\be))$ and \emph{coherent systems} of measures $\{M_N\in\Prob(\wt\Om_N(q,t;\al,\be))\}$. Here `coherent' means in turn that $M_{N+1}\wt\La^{N+1}_N=M_N$ for each $N\in\Z_{\ge1}$. The coherent systems are thus elements of the projective limit space $\varprojlim \Prob(\wt\Om_N(q,t;\al,\be))$. The bijection  is determined by the relations $M_N=M\wt\La^\infty_N$.

\begin{corollary}\label{cor7.A}
The stochastic links $\wt\La^{N+1}_N$ and $\wt\La^\infty_N$ are Feller links in the sense that the corresponding operators preserve continuous functions.
\end{corollary}

\begin{proof}
This follows from \eqref{eq7.G} and \eqref{eq7.I} combined with the following two facts. First, by the very definition, stochastic links induce contraction operators with respect to the supremum norm. Second, the Macdonald polynomials $P_{\nu\mid N}(-;q,t)$ span a dense subspace of $C(\wt\Om_N(q,t;\al,\be))$, while the Macdonald symmetric functions $P_\nu(-;q,t)$ span a dense subspace of $C(\wt\Om(q,t;\al,\be))$. 
\end{proof}

\subsection{Markov dynamics on $\wt\Om(q,t;\al,\be)$}

In addition to the quadruple $(q,t;\al,\be)$ we fix a pair $(\ga,\de)$ of parameters such that $\ga=\bar\de\in\C\setminus\R$. Recall that the notion of Feller semigroups was defined in section \ref{sect2.3}. 

Given $\la\in\Y$, we set
\begin{equation}\label{eq7.A1}
\m_\la(q,t;\alde):=-\, \sum_{i=1}^\infty\left(\frac{\ga\de q}{\al\be} t^{-i+1}(q^{\la_i}-1)+t^{\la_i-1}(q^{-\la_i}-1)\right).
\end{equation}
The sum is actually finite, because $q^{\la_i}-1=q^{-\la_i}-1=0$ for $i$ large enough.

As before, we treat $\Sym$ as a dense subspace of the real Banach space $C(\wt\Om(q,t;\al,\de))$, so that symmetric functions are viewed as continuous functions on the compact space $\wt\Om(q,t;\al,\de)$. 

\begin{theorem}\label{thm7.A}
{\rm(i)} On the space $C(\wt\Om(q,t;\al,\de))$, there exists a Feller semigroup $\{T^{q,t;\alde}_\infty(s)\}$, $s\ge0$,
uniquely determined by the property that the action of the semigroup operators on the big $q$-Jacobi symmetric functions is given by 
\begin{equation}\label{eq7.A}
T^{q,t;\alde}_\infty(s) \Phi_\la(-; q,t;\alde)= \exp(s\cdot\m_\la(q,t;\alde))\Phi_\la(-; q,t;\alde), 
\end{equation}
where $s\ge0$ and $\la\in\Y$.

{\rm(ii)} The measure $\MM_\infty^{q,t;\alde}$ on $\wt\Om(q,t;\al,\de)$, introduced in Proposition \ref{prop6.D}, is a stationary measure for $\{T^{q,t;\alde}_\infty(s)\}$, and it is a unique probability measure with this property. 
\end{theorem}

\begin{proof}
(i) The uniqueness statement is evident, because the big $q$-Jacobi symmetric functions form a basis in $\Sym$, which is dense in $C(\wt\Om(q,t;\al,\de))$. We proceed to the existence statement.

\emph{Step}\/ 1. For $N=1,2,3,\dots$, let $\D_N=\D_N^{q,t;\alde}$ stand for the $q$-difference operator on $\Sym(N)$ introduced in \eqref{eq5.B}, with the varying parameters 
\begin{equation}\label{eq7.B}
(\abcd)= (\al,\be;\ga t^{1-N}, \de t^{1-N}),
\end{equation}
as in \eqref{eq6.L}. By Theorem \ref{thm5.A}, $\D_N$ gives rise to a Feller semigroup $\{T_N^{q,t,\alde}(s)\}$ on the Banach space $C(\wt\Om(q,t;\al,\be))$. Below we use the shorthand notation $\{T_N(s)\}$. 

From \eqref{eq6.F} it follows that the  big $q$-Jacobi symmetric polynomials with the parameters \eqref{eq7.B} are eigenfunctions of $\D_N$. Moreover, comparing the expression \eqref{eq5.M} for the eigenvalues with \eqref{eq7.A1} and taking into account \eqref{eq7.B} we see that the eigenvalues are \emph{stable in} $N$:
\begin{multline}\label{eq7.C}
\D_N\, \varphi_{\la\mid N}(-;q,t; \al,\be;\ga t^{1-N}, \de t^{1-N})\\
=\m_\la(q,t;\alde)  \varphi_{\la\mid N}(-;q,t; \al,\be;\ga t^{1-N}, \de t^{1-N})
\end{multline}
for any $\la\in\Y(N)$.

Next, we claim that 
\begin{multline}\label{eq7.D}
T_N(s) \varphi_{\la\mid N}(-;q,t; \al,\be;\ga t^{1-N}, \de t^{1-N})\\
=\exp(s\cdot \m_\la(q,t;\alde))  \varphi_{\la\mid N}(-;q,t; \al,\be;\ga t^{1-N}, \de t^{1-N}))
\end{multline}
for any $s\ge0$  and $\la\in\Y(N)$. 
This conclusion is not immediate from \eqref{eq7.C}, because the pre-generator $\D_N$ is not bounded, so that we cannot simply write $T_N(s)=\exp(s\cdot\D_N)$. However, we may use instead the limit relation \eqref{eq2.A}. Here we  rely on the evident fact that 
$\D_N$ does not rise degree and hence preserves all finite-dimensional subspaces $\Sym(N)_{\le d}$ (see section \ref{sect6.2}). This allows us to apply the formula \eqref{eq2.A} directly, as the limit can be taken in a suitable invariant finite-dimensional subspace. 

For further use, it is convenient to rewrite \eqref{eq7.D} in the form
\begin{multline}\label{eq7.D1}
T_N(s) \dfrac{\varphi_{\la\mid N}(-;q,t; \al,\be;\ga t^{1-N}, \de t^{1-N})}{(t^N;q,t)_\la)}\\
=\exp(s\cdot \m_\la(q,t;\alde))   \dfrac{\varphi_{\la\mid N}(-;q,t; \al,\be;\ga t^{1-N}, \de t^{1-N})}{(t^N;q,t)_\la)}.
\end{multline}

\emph{Step}\/ 2. We claim that 
\begin{equation}\label{eq7.E}
\wt\La^{N+1}_N \dfrac{\varphi_{\la\mid N}(-;q,t; \al,\be;\ga t^{1-N}, \de t^{1-N})}{(t^N;q,t)_\la}=
\dfrac{\varphi_{\la\mid N+1}(-;q,t; \al,\be;\ga t^{-N}, \de t^{-N})}{(t^{N+1};q,t)_\la}.
\end{equation}
Indeed, the similar relation holds for the (renormalized) Macdonald polynomials, see \eqref{eq7.G}. But \eqref{eq6.H} shows that the transition coefficients between the two systems of polynomials are stable in $N$. This proves \eqref{eq7.E}

\emph{Step}\/ 3. Next, we claim that the semigroups $\{T_N(s)\}$ are compatible with the links $\wt\La^{N+1}_N$ in the sense that the following \emph{intertwining relation} holds:
\begin{equation}\label{eq7.F}
T_{N+1}(s)\circ\wt\La^{N+1}_N =\wt\La^{N+1}_N \circ T_N(s), \qquad s\ge0, \; N=1,2,\dots
\end{equation}
(here both sides are operators $C(\wt\Om_N(q,t;\al,\be))\to C(\wt\Om_{N+1}(q,t;\al,\be))$). 

Indeed, combining \eqref{eq7.E} with \eqref{eq7.D1} we see that the desired relation holds true when applied to the (renormalized) big $q$-Jacobi polynomials with indices $\la\in\Y(N)$, because the quantities $ \m_\la(q,t;\alde)$ do not depend on $N$. Since these polynomials form a basis of the dense subspace $\Sym(N)\subset C(\wt\Om_N(q,t;\al,\be))$ and all operators in \eqref{eq7.F} are bounded, the relation \eqref{eq7.F} holds on the whole space $C(\wt\Om_N(q,t;\al,\be))$. 

\emph{Step}\/ 4. The intertwining relation \eqref{eq7.F} allows us to apply the general formalism described in \cite[Proposition 2.4]{BO-2012}. Because we know that $\wt\Om(q,t;\al,\be)$ is the Feller boundary of the chain $\{\wt\Om_N(q,t;\al,\be), \La^{N+1}_N\}$ (see Proposition \ref{prop7.C}), this formalism says that there exists a unique Feller semigroup $\{T_\infty(s)\}$ on $C(\wt\Om(q,t;\al,\be))$ such that 
\begin{equation}\label{eq7.H}
T_\infty(s)\circ\wt\La^\infty_N=\wt\La^\infty_N\circ T_N(s), \qquad s\ge0, \; N=1,2,\dots\,.
\end{equation}
Here and below we write $\{T_\infty(s)\}$ instead of the more detailed notation $\{T^{q,t;\alde}_\infty(s)\}$.

\emph{Step}\/ 5. It remains to show that  the big $q$-Jacobi symmetric functions $\Phi_\la(-;q,t;\alde)$ are eigenfunctions of the operators $T_\infty(s)$, as stated in \eqref{eq7.A}. We know that the similar claim holds at each level $N$, for the renormalized $N$-variate big $q$-Jacobi polynomials, see \eqref{eq7.D1}. To deduce from this the desired equality \eqref{eq7.A} it suffices to prove that 
\begin{equation}\label{eq7.J}
\wt\La^\infty_N\, \dfrac{\varphi_{\la\mid N}(-;q,t; \al,\be;\ga t^{1-N}, \de t^{1-N})}{(t^N;q,t)_\la}=\Phi_\la(-;q,t;\alde), \qquad \forall \la\in\Y(N).
\end{equation}
Indeed, suppose for a moment that we already know \eqref{eq7.J}. Then, given $\la\in\Y$, we take $N$ so large that $\la\in\Y(N)$, apply both sides of \eqref{eq7.H} to the renormalized polynomial with index $\la$, and use \eqref{eq7.D1}. This gives us \eqref{eq7.A}.

\emph{Step}\/ 6. Finally, to prove \eqref{eq7.J} we observe that a similar relation holds for the renormalized $N$-variate Macdonald polynomials and Macdonald symmetric functions, see \eqref{eq7.I}. To pass from it to \eqref{eq7.J} we use the same argument as in step 2: the stability of the transition coefficients in \eqref{eq6.H}. 

This concludes the proof of (i).

(ii) By virtue of \eqref{eq7.A} and the fact that all eigenvalues $\m_\la(q,t;\alde)$ with $\la\ne\varnothing$ are nonzero, the stationarity  condition for a measure exactly means that it must be orthogonal to all functions $\Phi_\la(-;q,t;\alde)$ with $\la\ne\varnothing$. By virtue of Remark \ref{rem6.A}, $\MM^{q,t;\alde}$ is a unique probability measure satisfying this condition. 
\end{proof}

\subsection{Concluding remarks} 

1. It is easy to see that the operator $\D_\infty$ defined in \eqref{eq1.A} is a pre-generator of the semigroup $\{T_\infty(s)\}$, as stated in Theorem \ref{thm1.A}.

2. As in the context of Theorems \ref{thm3.A} and \ref{thm5.A}, the semigroup $\{T_\infty(s)\}$, being a Feller semigroup, gives rise to a Markov process on the space $\wt\Om(q,t;\al,\be)$. 

3. Our construction of the Feller semigroup $\{T_\infty(s)\}$ out of the Feller semigroups $\{T_N(s)\}$ relies on the general formalism of \cite{BO-2012}. The example from \cite[section 10]{BO-2013} shows that this formalism may sometimes lead to a deterministic process on the boundary of a chain. However, this is not the case in our situation.

4. One can show that the semigroups $\{T_N(s)\}$ approximate the semigroup $\{T_\infty(s)\}$ in the same sense as in \cite[Proposition 8.13]{BO-2013}. 

5. The stationary measures $\MM_\infty^{q,t;\alde}$ are of independent interest. In the special case $t=q$, they are determinantal measures, see \cite{GO} and \cite{CGO} (these papers focus on a different range of the parameters $(\alde)$, but the results are valid in our context, too). 

6. Throughout the paper we assumed that $c=\bar d\in \C\setminus\R$ and likewise $\ga=\bar\de\in\C\setminus\R$. However, the results remain valid for a larger range of parameters.

\bigskip

Krichever Center for Advanced Studies, Skoltech, Moscow,
Russia

HSE University, Moscow, Russia

olsh2007@gmail.com


\begin{thebibliography}{AA}

\bibitem{AA}
G. E. Andrews, R. Askey, Classical orthogonal polynomials. In: Polyn\^{o}mes Orthogonaux
et Applications (C. Brezinski, A. Draux, A. P. Magnus, P. Maroni, and A. Ronveaux, eds.),
Lecture Notes in Math. 1171, Springer-Verlag, New York, 1985, pp. 36-62.

\bibitem{Assiotis}
T. Assiotis, Hua-Pickrell diffusions and Feller processes on the boundary of the graph of spectra. Annales de l'Institut Henri Poincar\'e, Probabilités et Statistiques 56 (2020), no. 2, 1251-1283.

\bibitem{BG}
A. Borodin, V. Gorin, Markov processes of infinitely many nonintersecting random walks. Prob. Theory Rel. Fields 155 (2013), 935-997. 

\bibitem{BO-2012}
A. Borodin, G.  Olshanski, Markov processes on the path space of the Gelfand-Tsetlin
graph and on its boundary. J. Funct. Anal. 263 (2012), no. 1, 248-303.  

\bibitem{BO-2013}
A. Borodin, G. Olshanski, Markov dynamics on the Thoma cone: a model of time-dependent determinantal processes with
infinitely many particles. Electron. J. Probab. 18 (2013), no. 75, 1-43.

\bibitem{Cuenca}
C. Cuenca, Markov processes on the duals to infinite-dimensional Lie groups. Annales de l’Institut Henri Poincaré, Probabilités et Statistiques 54 (2018), no. 3, 1359-1407.

\bibitem{CGO}
C. Cuenca, V. Gorin, G. Olshanski, The elliptic tail kernel. Intern. Math. Research. Notices  vol. 2021 (2021), No. 19, 14922-14964.

\bibitem{CO}
C. Cuenca, G. Olshanski, Elements of the q-Askey scheme in the algebra of symmetric functions. 
Moscow Mathematical Journal 20 (2020), no. 4, 645-694.

\bibitem{D}
N. Demni, $\be$-Jacobi processes. Adv. Pure Appl. Math. 1 (2010), no. 3, 325-344. 

\bibitem{EK}
S. N. Ethier, T. G. Kurtz, Markov Processes. Characterization and Convergence. Wiley, 2005. 


\bibitem{GO}
V. Gorin, G. Olshanski, A quantization of the harmonic analysis on the infinite-dimensional
unitary group. J. Funct. Anal. 270 (2016), 375-418.

\bibitem{G-2009}
W. Groenevelt, The vector-valued big $q$-Jacobi transform. Constr. Approx. 29 (2009), 85-127.

\bibitem{HO}
G. Heckman, E. Opdam, Jacobi polynomials and hypergeometric functions
associated with root systems. In: Encyclopedia of Special Functions: The Askey-Bateman Project. Vol. 2. Multivariable Special Functions (T. H. Koornwinder and J. V. Stokman. eds.), Cambridge Univ. Press, 2021, chapter 8.  

\bibitem{Ito}
K. Ito, Essentials of Stochastic Processes. Translations of  Mathematical Monographs vol. 231. Amer. Math. Soc., 2006. 


\bibitem{Kall}
O. Kallenberg, Foundations of Modern Probability. Springer, 2010. 

\bibitem{KS}
R. Koekoek, R. F. Swarttouw, The Askey-scheme of hypergeometric orthogonal polynomials and its q-analogue, Report 98-17, Faculty of Technical Mathematics and Informatics, Delft University of Technology, 1998, http://aw.twi.tudelft.nl/~koekoek/askey/.

\bibitem{Koo-2014}
T. H. Koornwinder, Additions to the formula lists in ``Hypergeometric orthogonal polynomials and their q-analogues'' by Koekoek, Lesky and Swarttouw, arXiv:1401.0815.

\bibitem{Ligg}
T. M. Liggett, Continuous Time Markov Processes. Graduate Studies in Mathematics vol. 113. Amer. Math. Soc., 2010. 

\bibitem{Mac-1995}
I. G. Macdonald, Symmetric functions and Hall polynomials. 2nd ed., Oxford Univ. Press, 1995.

\bibitem{Meyer}
P.-A. Meyer, Probability and Potentials, Blaisdell, 1966.

\bibitem{Ok}
A. Okounkov, $BC$-type interpolation Macdonald polynomials and binomial formula for Koornwinder polynomials. Transf. Groups. 3 (1998), 181-207.

\bibitem{Ols-2017}
G. Olshanski, An analogue of the big q-Jacobi polynomials in the algebra of symmetric
functions. Funct. Anal. Appl. 51, no. 3 (2017), 204-220.

\bibitem{Ols-2021a}
G. Olshanski, Macdonald polynomials and extended Gelfand--Tsetlin graph. Selecta Mathematica
New Series 27 (2021), no. 2, paper 41.

\bibitem{Ols-2021b}
G. Olshanski, Macdonald-level extension of beta ensembles and large-N limit transition.
Commun. Math. Phys. 385 (2021), 595-631.

\bibitem{R}
E. M. Rains, $BC_n$-symmetric polynomials. Transf. Groups 10 (2005), 63-102. 

\bibitem{RR}
H. Remling, M. R\"osler, The heat semigroup in the compact Heckman--Opdam setting
and the Segal--Bargmann transform. Intern. Math. Res. Lett. 2011 (2011), No. 18, 4200-4225. 


\bibitem{S}
J. V. Stokman, Multivariable big and little $q$-Jacobi polynomials. SIAM J. Math. Anal. 28 (1997), No. 2, pp. 452--480. 

\bibitem{S-2000}
J. V. Stokman, On BC type basic hypergeometric orthogonal polynomials. Trans. Amer. Math. Soc. 352 (2000), 1527--1579.

\bibitem{SK}
J. V. Stokman,  T. H. Koornwinder, Limit transitions for BC type multivariable orthogonal polynomials. Canad. J. Math. 49 (1997), 374--405.




\end{thebibliography}
\end{document}